\documentclass[12pt]{amsart} \usepackage{amscd, amssymb}   
\usepackage[matrix, arrow]{xy}  
\usepackage{array} 

\usepackage{color}

\definecolor{darkgreen}{cmyk}{1,0,1,.2}
\definecolor{m}{rgb}{1,0.1,1}

\begin{document} 
{\tiny \today} 
\def\dach{\!\widehat{ \;\;\; }}
\def\A{\mathcal{A}} 
\def\G{\mathcal{G}} 
\def\SS{\mathcal{S}} 
\def\C{\mathcal{C}} 
\def\H{\mathcal {H}}
\def\d{\operatorname{d}}
\def\im{\operatorname{im}}
\def\iDelta{{\Delta\!\!\!\!\circ\,}}
\def\bDelta{\partial\Delta}
\def\const{\operatorname{const}}

\def\m{{\boldmath m}}
\def\l{{\boldmath l}}
\def\h{{\boldmath h}}
\def\x{{\boldmath x}}
\def\y{{\boldmath y}}
\def\ab{\operatorname{ab}}

\def\tg{\operatorname{tg}}

\def\Call{\mathcal C_{\operatorname{all}}}
\def\Csep{\mathcal C_{\operatorname{sep}}}
\def\Cnuc{\mathcal C_{\operatorname{nuc}}}
\def\Ccom{\mathcal C_{\operatorname{com}}}
\def\Ob{\operatorname{Ob}}
\def\Mor{\operatorname{Mor}} 
\def\Aut{\operatorname{Aut}} 
\def\cell{\operatorname{cell}} 
\def\Cell{\operatorname{Cell}} 
\def\incl{\operatorname{incl}} 
\def\sign{\operatorname{sign}} 
\def\Gx{X\rtimes G}  
\def\Bott{\operatorname{Bott}} \def\CC{\mathbb C} 
\def\ClV{C_{\!{}_{V}}} \def\comp{\operatorname{comp}} 
\def\ddS{\stackrel{\scriptscriptstyle{o}}{S}} 
\def\intW{\stackrel{\scriptscriptstyle{o}}{W}} 
\def\intU{\stackrel{\scriptscriptstyle{o}}{U}} 
\def\intT{\stackrel{\scriptscriptstyle{o}}{T}}
\def\intsi{\stackrel{\scriptscriptstyle{o}}{\sigma}}
\def\E{\mathcal E} 
\def\EE{\mathbb E} \def\EG{\mathcal{E}(G)} 
\def\Egx{\mathcal{E}(G)\times X} 
\def\EGG{\mathcal{E}(\mathcal{G})} \def\EGT{\mathcal{E}(G/G_0)}\def\Eg{\mathcal{E}(G)} 
\def\EGN{\mathcal{E}(G/N)} \def\EH{\mathcal{E}(H)} \def\F{\mathcal F} 
\def\I{{\operatorname{IND}}} \def\Ind{\operatorname{Ind}} 
\def\Id{\operatorname{Id}} 
\def\infl{\operatorname{inf}} \def\k{\operatorname{K}} 
\def\KK{\operatorname{KK}} \def\L{\mathcal L} \def\lk{\langle} 
\def\rk{\rangle} \def\NN{\mathbb N} \def\oplus{\bigoplus} 
\def\pt{\operatorname{pt}} \def\QQ{\mathbb Q} 
\def\res{\operatorname{res}} \def\RKK{\mathcal{R}\!\operatorname{KK}} 
\def\RR{\mathbb R} \def\sm{\backslash} \def\top{\operatorname{top}} 
\def\ZM{{\mathcal Z}M} \def\ZZ{\mathbb Z} \setcounter{section}{-1} 
\def\Inf{\operatorname{Inf}} \def\Ad{\operatorname{Ad}} 
\def\K{\mathcal K} \def\id{\operatorname{id}} 
\def\U{\mathcal U}
\def\V{\mathcal V}
\def\PU{\mathcal PU}
\def\BG{\operatorname{BG}}
\def\eps{\epsilon} 
\def\om{\omega}
\def\F{\mathcal{F}} 
\def\TT{\mathbb T}
\def\Flk{\F_{L,K}} 
\def\mg{\mu_{G,A}} 
\def\mgx{\mu_{\Gx,A}} 
\def\ts{\otimes} 
\def\ga{\gamma} 
\def\Br{\operatorname{Br}}
\def\Ab{\operatorname{Ab}}
\def\ev{\operatorname{ev}}
\def\la{\lambda} 
\def\oDelta{\stackrel{\circ}{\Delta}}
\def\CKK{\mathcal K\mathcal K}
\emergencystretch= 30 pt
\theoremstyle{plain} \newtheorem{thm}{Theorem}[section] 
\newtheorem{cor}[thm]{Corollary} \newtheorem{lem}[thm]{Lemma} 
\newtheorem{prop}[thm]{Proposition} \newtheorem{lemdef}[thm]{Lemma and 
 Definition} \theoremstyle{definition} 
\newtheorem{defn}[thm]{Definition}  
\newtheorem{defremark}[thm]{\bf Definition and Remark} 
\theoremstyle{remark} 
\newtheorem{remark}[thm]{\bf Remark} \newtheorem{ex}[thm]{\bf Example} 
\newtheorem{notations}[thm]{\bf Notations} 
\numberwithin{equation}{section} \emergencystretch 25pt 
\renewcommand{\theenumi}{\roman{enumi}} 
\renewcommand{\labelenumi}{(\theenumi)} \title[Fibrations with noncommutative fibers]{Fibrations with noncommutative fibers}   
\author[Echterhoff]{Siegfried 
 Echterhoff} 
\address{S. Echterhoff: Westf\"alische Wilhelms-Universit\"at M\"unster, 
 Mathematisches Institut, Einsteinstr. 62 D-48149 M\"unster, Germany} 
\email{echters@math.uni-muenster.de}
\author[Nest]{Ryszard Nest}
\address{R. Nest: Department of Mathematics, University
 of Copenhagen, Universitetsparken 5, DK-2100 Copenhagen, Denmark}
 \email{rnest@math.ku.dk}
\author[Oyono-Oyono]{Herve Oyono-Oyono}
\address{H. Oyono-Oyono: Universit\'e de Blaise Pascal, Clermont-Ferrand} 
\email{oyono@math.univ-bpclermont.fr} 
\thanks{This work was partially supported by the Deutsche Forschungsgemeinschaft
(SFB 478)}

\begin{abstract} We study an analogue of fibrations of topological spaces with
the homotopy lifting property in the setting of $C^*$-algebra bundles.
We then derive an analogue of the Leray-Serre spectral sequence to
compute the $K$-theory of the fibration in terms of the
{cohomology} of the  
base and the $K$-theory of the fibres. We present many examples which show
that fibrations with noncommutative fibres appear in abundance in nature.
\end{abstract} 
\maketitle  


\section{Introduction}

In recent years the study of the topological properties of C*-algebra bundles 
plays a more and more prominent r\^ole in the field of Operator algebras.
The main reason for this is two-fold: on one side there are many important 
examples of C*-algebras which do come with a canonical bundle structure. On the other side,
the study of C*-algebra bundles over a locally compact Hausdorff base space 
$X$ is the natural next step in classification theory, after the far reaching results
which have been obtained in the classification of simple C*-algebras.
To fix notation, by a {\em C*-algebra bundle} $A(X)$ over $X$ we shall simply mean 
a $C_0(X)$-algebra in the sense of Kasparov (see \cite{Kas2}): it is a C*-algebra 
$A$ together with a non-degenerate $*$-homomorphism 
$$\Phi: C_0(X)\to ZM(A),$$
where $ZM(A)$ denotes the center of the multiplier algebra $M(A)$ of $A$.
For such $C_0(X)$-algebra $A$, {the {\em fibre}   over  $x\in X$
 is then 
 $A_x=A/I_x$, where $$I_x=\{\Phi(f)\cdot a;\, a\in A \text{ and } f\in
 C_0(X)\text{ such that } f(x)=0\},$$ and the canonical quotient map
 $q_x:A\to A_x$ is called the {\em evaluation map}  at $x$}. We shall often write $A(X)$ to indicate the given $C_0(X)$-structure of $A$. 
We shall recall the basic constructions and properties of $C_0(X)$-algebras in the 
preliminary section below. We refer to \cite{ENO} for further notations concerning 
$C_0(X)$-algebras.

The main problem when studying bundles from the topological point of view is to 
provide good  topological invariants which help to understand the local and global
structure of the bundles. A good example is given by the class of separable
continuous-trace C*-algebras,
which are, up to Morita equivalence, just the section algebras of locally trivial bundles
over $X$ with fibres the compact operators $\K\cong \K(l^2(\NN))$.
Using the standard classification 
of fibre bundles, these algebras (or rather the underlying bundle structure) are classified 
up to Morita equivalence by a corresponding Dixmier-Douady class in $\check H^3(X,\ZZ)$.
Another interesting class of examples are the non-commutative principle torus bundles, 
which have been studied by the authors in \cite{ENO}. 
A basic example of a non-commutative principal  2-torus bundle
is given by the C*-algebra $C^*(H)$ of
the discrete rank 3 Heisenberg group $H$, which has a canonical structure of a C*-algebra
bundle over the circle $\TT$ where the  fibre $A_z$ over $z\in \TT$ is  the non-commutative 
2-torus $A_\theta$ if $z=e^{2\pi i\theta}$.
This shows in particular, that such bundles are in general far away from being section algebras 
of locally trivial C*-algebra bundles (but see \cite[\S 2]{ENO} for a classification based on classical methods).

The main purpose of \cite{ENO} was the 
 study of the $K$-theoretic properties of the principle non-commutative $\TT^n$-bundles
  after forgetting the $\TT^n$-actions.
 Using Kasparov's $\RKK(X;\cdot,\cdot)$-theory as the version of $K$-theory which 
 is probably most adapted to the study C*-algebra bundles, we show in 
 \cite[Corollary 3.4]{ENO} that the non-commutative $\TT^n$-bundles are always locally $\RKK$-trivial, which
 means that for each $x\in X$ there exists a neighbourhood $U$ of $x$ such that 
 the restriction $A(U)$ of $A(X)$ to $U$ is $\RKK(U;\cdot,\cdot)$-equivalent to 
 $C_0(U\times \TT^n)$.  As usual, the global picture is much more difficult.
 Using the local $\RKK$-triviality we show in \cite{ENO} that to each non-commutative torus bundle 
 $A(X)$ we may associate a corresponding bundle of $K$-theory groups which comes 
 equipped with a canonical action of the fundamental group $\pi_1(X)$ of the base $X$.
Using this associated group bundle allows us to obtain at least a partial classification result
up to $\RKK$-equivalence (see \cite[Theorem 7.5]{ENO}).

In this paper we want  to extend the studies of \cite{ENO} from a more general perspective.
Indeed, we are interested in C*-algebra bundles which are non-commutative analogues of 
classical fibrations in topology which satisfy certain weak versions of 
the homotopy lifting property. Indeed, the important point 
implied by the  homotopy lifting property in classical topology is 
that for any fibration $q:Y\to X$ 
with this property, the space $Y$ looks, in a topological sense,  locally
like a product space
$U\times F$. The phrase ``in a topological sense'' means that any homotopy invariant
(co-)homology theory
cannot differentiate between $p^{-1}(U)$ and $U\times F$. 

Since it seems to be impossible to rephrase the homotopy lifting property 
in the non-commutative setting, we shall give a definition of this property in dependence of a
given (co-)homology theory on the category of C*-algebras.
For example,
a (section algebra of a) C*-algebra bundle $A(X)$ over $X$ is called 
a {\em $K$-fibration} if for any positive integer $p$, for any $p$-simplex $\Delta^p$
and for any continuous map $f:\Delta^p\to X$  the pull-back $f^*A(\Delta^p)$ of $A(X)$ 
via $f$ is $K$-theoretically trivial in the sense that the evaluation homomorphism
$$q_v:f^*A(\Delta^p)\to A_{f(v)}$$
induces an isomorphism of $K$-theory groups $K_i(f^*A(\Delta^p))\cong K_i(A_{f(v)})$
for all $v\in \Delta^p$. In a similar way we can define $\KK$-fibrations, $\RKK$-fibrations
or $h$-fibrations, when $(h_n)_{n\in \ZZ}$ (resp $(h^n)_{n\in \ZZ}$) is any
given (co-)homology theory on a suitable category of C*-algebras.
In case of $K$-theory, the strongest notion will be that of an $\RKK$-fibration (which implies 
that such bundles are automatically $\KK$- and $K$-fibrations) and we shall
indicate that there exist many  natural examples of such fibrations. 
For instance, the principle non-commutative torus bundles of \cite{ENO} are 
always $\RKK$-fibrations. 

The main result of this paper will be the proof of a non-commutative analogue 
of the Leray-Serre spectral sequence for general $h$-fibrations.
Indeed, if $(h_n)_{n\in \ZZ}$ (resp. $(h^n)_{n\in \ZZ}$) is any
given (co-)homology theory on a suitable category of C*-algebras,
and if $A(X)$ is an $h$-fibration over the geometric realisation of a simplicial complex $X$,
we show that we can associate to $A$ the group bundle $\mathcal H=\{h_q(A_x): x\in X\}$
which carries a canonical action of $\pi_1(X)$. The Leray Serre spectral sequence 
for $A(X)$ then converges to $h(A(X))$ and  has
(co-)homology groups $H^p(X, \mathcal H_q)$ as $E_2$-terms.

Thus, at least in principle we can use the spectral sequence 
for computation of the $K$-theory groups of any $K$-fibration 
$A(X)$. In particular this applies to the {principal} non-commutative $\TT^n$-bundles
as studied in \cite{ENO}. The spectral sequence also serves as an obstruction 
for $\RKK$-equivalence of two bundles $A(X)$ and $B(X)$---any such equivalence
induces an isomorphism between the respective spectral sequences. 
It is certainly an interesting question to what extend the converse might hold,
at least in case where $A(X)$ and $B(X)$ are $\RKK$-fibrations (or 
locally $\RKK$-trivial).
{In a final section we apply the spectral sequence to the study of the 
noncommutative torus bundles of \cite{ENO} and show that it gives 
the missing tool for deciding which noncommutative torus bundles are 
globally $\RKK$-trivial. We further give an explicit computation of the spectral 
sequences in the case of non-commutative $2$-torus bundles over $\TT^2$. 
The results show that there are noncommutative principle torus bundles 
with isomorphic spectral sequences for which we do not know at this point 
whether they are
$\RKK$-equivalent.

\section{Some preliminaries}

\subsection{Homology theories on $C^*$-algebras.}
Let $\Call$ denote 
the category of all $C^*$-algebras with $*$-homomorphisms as morphisms.
By a {\em good subcategory} of $\Call$ we shall understand any subcategory 
$\C$ of $\Call$ with $\CC\in \C$ and which is closed under taking ideals, quotients, extensions and 
suspension in the sense that if $A\in \Ob(\C)$, then $SA:=C_0(\RR,A)\in \Ob(\C)$.
Moreover, for simplicity, 
we shall assume that $\Mor_{\C}(A,B)=\Mor_{\Call}(A,B)$ for all $A,B\in \C$
and that if $A\cong B$ in $\Call$ and $A\in \C$, then $B\in\C$.
In many cases considered below, the above assumption on the morphisms 
could probably be weakened to the assumptions
given in \cite[21.1]{Black}, but we don't want to bother with this extra generality.
Standard examples of good subcategories of $\Call$ are given by the category
$\Csep$ of separable $C^*$-algebras, the category $\Cnuc$ of nuclear $C^*$-algebras
or the category $\Ccom$ of commutative $C^*$-algebras. Following \cite[21.1]{Black}
we now give the following

\begin{defn}\label{homtheory}
A {\em homology theory} on a good subcategory $\C$ of $\Call$ is a sequence $\{h_n\}_{n\in\ZZ}$
of covariant functors $h_n$ from $\C$ to the 
category $\Ab$ of abelian groups satisfying the following 
axioms:
\begin{itemize}
\item[(H)] If $f_0,f_1:A\to B$ are homotopic, then $f_{0,*}=f_{1,*}:h_n(A)\to h_n(B)$ for all $n\in \ZZ$.
\item[(LX)] If $0\to J\stackrel{i}{\to} A\stackrel{q}{\to} B\to 0$ is a short exact sequence in 
$\C$, then for each $n\in \ZZ$
there are connecting maps $\partial_n:h_n(B)\to h_{n-1}(J)$, natural with respect to morphisms 
of short exact sequences,  making exact the following
long sequence
$$\cdots \stackrel{\partial_{n+1}}{\longrightarrow} h_n(J)
\stackrel{i_*}{\longrightarrow} h_n(A)
\stackrel{q_*}{\longrightarrow} h_n(B)
\stackrel{\partial_{n}}{\longrightarrow} h_{n-1}(J)
\stackrel{i_*}{\longrightarrow} \cdots
$$
\end{itemize}
Similarly, we define a {\em cohomology theory} on a good subcategory as a sequence 
$\{h^n\}_{n\in\ZZ}$
of contravariant functors $h^n:\C\to\Ab$ which satisfy  the obvious reversed axioms
(e.g. see \cite[21.1]{Black}).
\end{defn}

\begin{remark}\label{rem-homtheory}
{\bf (1)} It follows from these axioms that all $h_n:\C\to\Ab$ are additive in  the sense
that 
$$h_n(A_1\oplus A_2)=h_n(A_1)\oplus h_n(A_2)$$
and that $h_n(A)=\{0\}$ if $A$ is 
a contractible $C^*$-algebra. Since $CA:=C\big((-\infty,\infty],A\big)$ is contractible, it follows from 
(LX) applied to the short exact sequence 
$$0\to SA\to CA\to A\to 0$$
that 
$h_{n+1}(A)= h_n(SA)$ (resp. $h^{n-1}(A)=h^n(SA)$) for all $A\in \C$.
\\
{\bf (2)} A covariant (resp. contravariant) functor $F:\Csep\to \Ab$ is called {\em stable} if 
$$i_p:A\to A\otimes \K;\;
i_p(a)=a\otimes p$$
induces an isomorphism $i_{p,*}:F(A)\stackrel{\cong}{\to}F(A\otimes \K)$
(resp. $i_{p}^*:F(A\otimes \K)\stackrel{\cong}{\to}F(A)$)
for every one-dimensional projection $p\in \K$. It is shown in \cite[Corollary 22.3.1]{Black}
(the result is originally due to Cuntz \cite{Cu})
that every stable (co-)homology theory $\{h_n\}$ (resp. $\{h^n\}$) on $\Csep$ 
satisfies Bott-periodicity $h_{n+2}(A)=h_n(S^2A)\cong h_n(A)$ (resp.
$h^{n+2}(A)\cong h^n(A)$). Hence, every stable (co-)homology theory on $\Csep$ 
is $\ZZ/2\ZZ$-graded and the long exact sequence (LX) then becomes a cyclic six-term 
exact sequence.
\\
{\bf (3)}  A homology theory 
$\{h_n\}$ (resp.$ \{h^n\}$) is called
{\em $\sigma$-additive} (resp. {\em $\sigma$-multiplicative}) if $h_n(A)=\oplus_{i\in I}h_n(A_i)$ 
(resp. $h_n(A)=\prod_{i\in I} h_n(A_i)$) whenever 
$A\in \C$ is a countable direct sum  of objects $A_i\in \C$, $i\in I$, 
and similar for cohomology theories.

\noindent
{\bf (4)} The main example 
of a homology theory on $\Call$ (or any good subcategory $\C$ of $\Call$)
is given by $K$-theory, and  $K$-homology serves as the main example for a cohomology theory
on $\Call$. Note that $K$-theory is $\sigma$-additive and $K$-homology is $\sigma$-multiplicative.
\end{remark}

Assume now that $\C$ is a good subcategory of $\Call$ and suppose that
$A(X)\in \C$ is a $C_0(X)$-algebra. 
In what follows we write $\Delta^p=<v_0,\ldots,v_p>$
for the standard $p$-simplex {with vertices $v_0,\ldots,v_p$}. It follows  from the properties of a good 
subcategory of $\Call$ that if $A\in \C$ and $f:\Delta^p\to X$ is any continuous map,
then $f^*A= \big(C(\Delta^p)\otimes A\big)/I_f$, with $I_f$ is a suitable ideal in 
$C(\Delta^p)\otimes A$, is again an object in $\C$.
In particular, all fibers $A_x$ for $x\in X$ are in $\C$.
The following definition is motivated by the notations and 
results presented in \cite[Chapter I]{D}:

\begin{defn}\label{defn-fibrations}
Suppose that $\C$ is a good subcategory of $\Call$ and that
$\{h_n\}$ is a homology theory on $\C$ (resp. $\{h^n\}$ is a cohomology 
theory on $\C$). Suppose further that $A=A(X)$ is a $C_0(X)$-algebra in $\C$.
Then
\begin{enumerate}
\item $A(X)$ is called an {\em $h$-fibration} 
if for all continuous maps $f:\Delta^p\to X$ and for every {point} $v\in \Delta^p$,
the quotient map $q_v:f^*A \to A_{f(v)}$ induces an isomorphism 
$q_{v,*}:h_n(f^*A)\to h_n(A_v)$ (resp. $q_{v}^*:h^n(A_{f(v)})\to h^n(f^*A)$).
\item If $\C= \Csep$, then 
$A(X)$ is  called a {\em $KK$-fibration}, if for all continuous maps 
$f:\Delta^p\to X$ and for every element $v\in \Delta^p$
the quotient map $q_v:f^*A \to A_{f(v)}$
is a $KK$-equivalence.
\item If $\C=\Csep$, then $A(X)$ is called an {\em $\RKK$-fibration}, if 
$f^*A$ is $\RKK(\Delta^p;\cdot,\cdot)$-equivalent to $C(\Delta^p,A_{f(v)})$ for any
continuous map $f:\Delta^p\to X$ and for any element $v$ of $\Delta^p$.
\end{enumerate}
\end{defn}

\begin{remark}\label{rem-fibration}
{\bf (1)} Any $\RKK$-fibration is a $\KK$-fibration. This follows from the fact 
that if $\frak x\in \RKK(\Delta^p; C(\Delta^p,A_{f(v)}), f^*A)$ is an $\RKK$-equivalence,
then we get the following commutative diagram in $\KK$:
$$
\begin{CD} C(\Delta^p,A_{f(v)})  @>\frak x>\cong> f^*A\\
@Vq_v V\cong V      @VVq_v V\\
A_{f(v)}  @>\cong> \frak x(v)> A_{f(v)},
\end{CD}
$$
where all arrows except of the {right} vertical one are known to be isomorphisms in $\KK$.
But then all arrows are $\KK$-equivalences. We shall formulate below a partial converse
of this easy observation, which follows from  a result of D\^ad\^arlat.
\\
{\bf (2)} It is a direct consequence of \cite[Corollary 22.3.1]{Black} that
if $A(X)$ is a $KK$-fibration, then  $A(X)$ is an
$h$-fibration for 
any stable (co-)homology theory $\{h_n\}$ (resp.$ \{h^n\}$) on $\Csep$.
\\
{\bf (3)} Every locally trivial C*-algebra bundle $A(X)$ is an $\RKK$-fibration.
This follows  from the fact that a pull-back of a locally trivial C*-algebra bundle 
is again locally trivial, and that any locally trivial bundle over a contractible space is trivial
(e.g., see \cite{Huse}).
\\
{\bf (4)} All non-commutative  {principal} $n$-tori as considered in 
\cite{ENO} are $\RKK$-fibrations. This follows  from 
\cite[Proposition 3.1]{ENO}.
\\
{\bf (5)} Being an $h$-fibration (resp. $KK$-fibration,
resp. $\RKK$-fibration)
is preserved by taking pull-backs inside $\C$. This follows from the fact that
if $A$ is a $C_0(X)$-algebra in $\C$
and $g:Y\to X$ is any continuous map
such that $g^*(A)\in \C$, and if $f:\Delta^p\to Y$ is any continuous map, then
$f^*(g^*(A))=(g\circ f)^*(A)$, and hence evaluation at any vertex induces isomorphisms 
in $h$-theory.
\\
{\bf (6)} Being a $KK$-fibration is preserved under taking maximal 
tensor products with arbitrary separable $C^*$-algebras 
and by minimal tensor products with separable  exact $C^*$-algebras.
This follows from the fact that taking maximal or minimal tensor products
of a $KK$-equivalence $\frak x\in KK_0(C, D)$ with a 
fixed $C^*$-algebra $B$ gives a $KK$-equivalence 
$$\frak x\otimes B\in KK_0(C\otimes_{(\max)}B,
D\otimes_{(\max)}B).$$
Similar statements hold for $\RKK$-fibrations.
\end{remark}

In what follows next we want to show that in many situations being 
a $\KK$-fibration is actually equivalent to being an $\RKK$-fibration.
Recall that a C*-algebra bundle (i.e., a 
$C_0(X)$-algebra) $A(X)$ is called a {\em continuous C*-algebra bundle} if for all 
$a\in A$ the map $x\mapsto \|a_x\|$ is a continuous function on $X$.
We need the following deep 
theorem of D\^ad\^arlat (see \cite[Theorem 1.1]{Da}). 

\begin{thm}\label{thm-Dadarlat}
Let $X$ be a compact metrizable finite dimensional space and let $A(X)$ and $B(X)$ be 
separable nuclear continuous C*-algebra bundles over $X$.
Suppose further 
that $\sigma\in \RKK(X; A(X),B(X))$ is such that $\sigma(x)\in \KK(A_x,B_x)$ is invertible for all $x\in X$.
Then $\sigma$ is invertible in  $\RKK(X; A(X), B(X))$.
\end{thm}

As a direct corollary we get the partial converse to the observation made in item (2) of 
Remark \ref{rem-fibration}:

\begin{cor}\label{converse}
Suppose that $A(X)$ is a separable nuclear continuous C*-algebra  bundle over some 
locally compact space $X$. Then  $A(X)$ is a $\KK$-fibration, if and only if it is an 
$\RKK$-fibration.
\end{cor}
\begin{proof} Since every $\RKK$-fibration is a $\KK$-fibration by item (1) of 
Remark \ref{rem-fibration} we only have to show the converse. 
Write $\Delta:=\Delta^p$ and let $f:\Delta\to X$ be any continuous map.
Since $A$ is a $\KK$-fibration, there exists the inverse $q_v^{-1}\in \KK(A_{f(v)}, f^*A)$ of 
the evaluation map $q_v$. Consider the image of $q_v^{-1}$ under the composition of maps
$$
\begin{CD}
\KK(A_{f(v)}, f^*A) @>\sigma_{\Delta, C(\Delta)}>> 
 \RKK(\Delta; C(\Delta)\otimes A_{f(v)},C(\Delta)\otimes  f^*A)\\
 @>\mu_*>> 
\RKK(\Delta; C(\Delta, A_{f(v)}),f^*A),
\end{CD}
$$
where $\mu: C(\Delta)\otimes f^*A\to f^*A; \mu(g\otimes a)=g\cdot a$ is the multiplication 
homomorphism.  If we evaluate this class at a point $w\in \Delta$, we obtain the class
$(q_v^{-1})\otimes q_w\in \KK(A_{f(v)}, A_{f(w)})$, which is invertible since $A$
is a $\KK$-fibration.  Hence the result follows from D\^ad\^arlat's theorem.
\end{proof}

Another interesting problem is the relation between locally $\RKK$-triviality, which 
was discussed in \cite{ENO} in connection  with non-commutative torus bundles and 
the $\RKK$-fibrations considered here. Let us recall that a C*-algebra bundle 
$A(X)$ is called {\em locally $\RKK$-trivial}, if for every $x\in X$ there exists a neighbourhood 
$V$ of $x$ such that the restriction $A(V)$ of $A$ to $V$ is $\RKK(V;\cdot,\cdot)$-equivalent
to $C_0(V, A_x)$. We have seen in \cite{ENO} that all {principal} non-commutative torus bundles 
are locally $\RKK$-trivial.
The proof of the following proposition is then straightforward.

\begin{prop}\label{prop-loceucl}
Suppose that $X$ is  locally euclidean, i.e. every $x\in X$ has 
a neighbourhood $U$ which is homeomorphic to an open ball in some $\RR^n$.
Then, if   $A(X)$ is an $\RKK$-fibration it follows  that $A(X)$ is locally $\RKK$-trivial.
\end{prop}

A bit surprisingly, the converse of the above proposition seems to 
be much more complicated. We shall obtain it later as a corollary of another
remarkable theorem of D\^ad\^arlat (see \cite[Theorem 2.5]{Da}), which
states that every separable and nuclear continuous C*-algebra bundle
over some compact metrizable space $X$ is $\RKK(X;\cdot,\cdot)$-equivalent 
to  a continuous bundle of simple Kirchberg algebras, i.e., each fibre is a separable 
nuclear purely infinite C*-algebra. As a direct consequence we get

\begin{prop}\label{prop-Kirchberg}
Suppose that $A(X)$ is a separable nuclear continuous C*-algebra bundle 
over the compact metrizable finite dimensional space $X$ such that  $A(X)$  is locally $\RKK$-trivial.
Then $A(X)$ is $\RKK$-equivalent to a locally trivial bundle of stable Kirchberg algebras.
\end{prop}
\begin{proof}
By D\^ad\^arlat's theorem, we may assume that $A(X)$ is a C*-algebra bundle
of simple Kirchberg algebras, and by stabilizing this bundle, we may assume
that all fibers are stable. If $A(X)$ is locally $\RKK$-trivial, we can find for
each $x\in X$ a compact neighbourhood $V_x$ such that $A(V_x)\sim_{\RKK} C(V_x, A_x)$.
It is then a consequence of \cite[Theorem 2.7]{Da} that this equivalence is 
actually realized by an isomorphism $A(V_x)\cong C(V_x,A_x)$ of C*-algebra bundles
over $V_x$.
\end{proof}

As a corollary we get

\begin{cor}\label{cor-trivial}
If $A(X)$ is a separable nuclear continuous field of C*-algebras over a locally compact space 
$X$. If $A(X)$ is locally $\RKK$-trivial, then $A(X)$ is an $\RKK$-fibration.
\end{cor}
\begin{proof}
If $f:\Delta^p\to A(X)$ is any continuous map, the pull-back $f^*A(\Delta^p)$ satisfies
all requirements of the above proposition. Since $\Delta^p$ is contractible, every 
locally trivial bundle over $\Delta^p$ is trivial. Thus it follows from the proposition
that $f^*A(\Delta^p)$ is $\RKK$-equivalent  to a trivial bundle.
\end{proof}

\section{Examples}\label{sec-ex}
In this section  we want to show that $K$-fibrations and $\KK$-fibrations 
do appear quite often in nature. We already mentioned above that all locally
trivial C*-algebra bundles are $\RKK$-fibrations. 
Since being an $\RKK$-fibration is stable under $C_0(X)$-linear Morita equivalence,
this implies also that all continuous-trace C*-algebras with spectrum $X$ are 
$\RKK$-fibrations. Although these classes of C*-algebra bundles are certainly interesting,
it would probably  not give enough motivation for a general study of fibrations 
as we do in this paper. 

A class of interesting algebras which are, in general, far away from being 
locally trivial bundles of C*-algebras are the non-commutative principal 
torus bundles as studied in \cite{ENO}, and we already pointed out that 
all of them are $\RKK$-fibrations.  Recall that the principal non-commutative torus bundles 
are, by definition, crossed products of the form $C_0(X,\K)\rtimes \ZZ^n$, where 
$\ZZ^n$ acts fibre-wise on the trivial bundle $C_0(X,\K)$. We shall now see that,
with the help of the Baum-Connes conjecture, one can construct many other examples 
of $\RKK$-, $\KK$-, or $K$-fibrations via a similar crossed product 
construction.

Suppose that $A$ is a C*-algebra bundle 
and $\alpha:G\to \Aut(A)$ is 
any $C_0(X)$-linear action of the locally compact group $G$ on $A$, i.e.,
we have 
$$\alpha_s(f\cdot a)=f\cdot \alpha_s(a)\quad\text{for all $s\in G, f\in C_0(X)$, and $a\in A$}.$$
(We simply write $f\cdot a$ for $\Phi(f)a$ if $\Phi:C_0(X)\to ZM(A)$ is the $C_0(X)$-structure
map of the bundle).
Then $\alpha$ induces actions $\alpha^x:G\to \Aut(A_x)$ on the fibres $A_x$ 
via $\alpha^x_s(a+I_x)=\alpha_s(a)+I_x$. 
The full and reduced crossed products $A\rtimes_{(r)}G$ have canonical structures as 
$C_0(X)$-algebras via the composition of the given $C_0(X)$-structure $\Phi:C_0(X)\to ZM(A)$
of $A$ with the canonical embedding $M(A)\to M(A\rtimes_{(r)}G)$. 
For the full crossed product $A\rtimes G$, the fibre over $x\in X$ is then given by
the full crossed product $A_x\rtimes G$, which follows from the exactness 
of full crossed with respect to  
short exact sequences of $G$-algebras.
For the reduced crossed products the situation 
can be more complicated. However, if 
$G$ is exact in the sense of Kirchberg and Wassermann
(which is true for a large class of groups---see \cite{KW}), then 
the fibre of $A\rtimes_rG$ over $x\in X$ is $A_x\rtimes_rG$.

Note also that if $f:Y\to X$ is any continuous map, and if $\alpha:G\to \Aut(A)$ is a 
$C_0(X)$-linear action of $G$ on $A$, then we get a $C_0(Y)$-linear pull-back action 
$f^*(\alpha):G\to \Aut(f^*A)$ given on elementary tensors $g\otimes a\in f^*A=C_0(Y)\otimes_{C_0(X)}A$ by the formula
$$f^*(\alpha)_s(g\otimes a):= g\otimes \alpha_s(a).$$
It is then easily checked (e.g. see \cite{EW2}), that 
$f^*A\rtimes G\cong f^*(A\rtimes G)$ as $C_0(Y)$-algebras and
$f^*A\rtimes_{r}G\cong f^*(A\rtimes_rG)$ if $G$ is exact.

In what follows next, we want to give some conditions which imply that 
the $C_0(X)$-algebras $A\rtimes G$ and $A\rtimes_rG$ are either
$K_*$-fibrations, $\KK$-fibrations, or even $\RKK$-fibrations.
As the basic tool for this we shall use the Baum-Connes conjecture for $G$.
Recall that for any $G$-algebra $A$, the {\em topological $K$-theory} of $G$ with coefficient
$A$ is defined as 
$$K_*^{\top}(G;A)=\lim_Z\KK_*^G(C_0(Z), A),$$
where $Z$ runs through the $G$-compact subspaces of a universal proper $G$-space 
$\EG$. In \cite{BCH}, Baum, Connes and Higson constructed an assembly map
$$\mu_A: K_*^{\top}(G;A)\to K_*(A\rtimes_rG)$$
and they conjectured that this map should always be an isomorphism of groups.
Although this conjecture turned out to be false in general (e.g., see \cite{HLS}), 
the conjecture has been shown to be true for very  large classes of groups
including all amenable and, more general,  a-$T$-menable groups 
(see \cite{HK}).
In what follows, if $A$ is a fixed $G$-algebra, we shall say that {\em $G$ satisfies BC for $A$}
if the  map is an isomorphism for this special $G$-algebra $A$.

{A-$T$-menable groups satisfy in fact a stronger version of the Baum-Connes
conjecture, which can be stated as follows.} Recall that
a $G$-algebra $D$ is said to be a {\em proper} $G$-algebra, if $D$ is a $C_0(Z)$-algebra 
for some proper $G$-space $Z$ in such a way that the structure map $\Phi:C_0(Z)\to ZM(D)$
is $G$-equivariant. A group $G$ is said to have a {\em $\gamma$-element}
if there exists an element $\gamma_G\in \KK^G(\CC,\CC)$ and a proper $G$-algebra $D$ 
such that  $\gamma_G$ can be written as 
a Kasparov product $\beta\otimes_D\delta$ for some $\beta\in \KK^G(\CC, D)$ and 
$\delta\in \KK^G(D,\CC)$, and such that the restriction $\res_K^G(\gamma)=1
\in \KK^K(\CC,\CC)$ for all compact subgroups $K$ of $G$. 
If $G$ has a $\gamma$-element as above,
then, by work of Kasparov and Tu \cite{Kas2, Tu} (extended in  
\cite[Theorem 1.11]{CEO2} to the weaker notion of a $\gamma$-element used here)
the Baum-Connes assembly map is known to be split injective with 
image $\mu_A(K_*^{\top}(G;A))=\gamma_G\cdot K_*(A\rtimes_rG)$.
We say that $G$ satisfies the {\em strong Baum-Connes conjecture}
if  {$\gamma_G=1_G$ in $\KK_0^G(\CC,\CC)$}.
By the results of Higson and Kasparov in \cite{HK}, every a-$T$-menable group satisfies
the strong Baum-Connes conjecture. It is clear from the above discussion
that every group $G$ which satisfies the strong Baum-Connes conjecture 
satisfies BC for all $G$-algebras $A$.

\begin{prop}\label{prop-K}
Suppose that $A$ and $B$ are $G$-algebras and that $q\in \KK^G(A,B)$.
Let $j_G^{(r)}(q)\in \KK_0(A\rtimes_{(r)}G, B\rtimes_{(r)}G)$ denote the descent of $q$
for the full (resp. reduced) crossed products. For every compact subgroup $K$ of $G$ let
$$\varphi_K: K_*(A\rtimes K)\to K_*(B\rtimes K);\;\;\varphi_K(x)=x \otimes j_K(\res_K^G([q])).$$
Then the following are true:
\begin{enumerate}
\item If $G$ satisfies BC for $A$ and $B$ and if
$\varphi_K$ is an isomorphism for every compact subgroup 
$K$ of $G$, then $\cdot\otimes j_G^r(q):K_*(A\rtimes_rG)\to K_*(B\rtimes_rG)$ is an isomorphism.
\item If $G$ satisfies the strong Baum-Connes conjecture and if 
$\varphi_K$ is an isomorphism for every compact subgroup 
$K$ of $G$, then
$$\cdot\otimes j_G^{(r)}(q):K_*(A\rtimes_{(r)} G)\to K_*(B\rtimes_{(r)} G)$$ is an isomorphism for the full and reduced crossed products.
\item If $G$ satisfies the strong Baum-Connes conjecture 
and if $j_K(\res_K^G(q))$
is a $\KK$-equivalence between $A\rtimes K$ and $B\rtimes K$ for all compact subgroups
$K$ of $G$, then 
$j_G^{(r)}(q)$ is a $\KK$-equivalence between $A\rtimes_{(r)} G$ and $B\rtimes_{(r)} G$,
for the full and reduced crossed products.
\end{enumerate}
\end{prop}
\begin{proof} 
Since $G$ satisfies BC for $A$ and $B$ item (i) follows if we can show that 
taking Kasparov product with $q$ induces an isomorphism from $K_*^{\top}(G;A)$ 
to $K_*^{\top}(G;B)$. But since all $\varphi_K$ are isomorphisms, this follows from 
\cite[Proposition 1.6]{ELPW}.

The proof of (ii) is a  consequence of (i) and the fact that  the strong Baum-Connes conjecture implies 
the Baum-Connes conjecture for all $G$-algebras and it implies also that $G$ is $K$-amenable 
which implies that
 the regular representation $L: A\rtimes G\to A\rtimes_rG$ induces an isomorphism 
in $K$-theory \cite{Tumoy}.

Finally, the proof of (iii) follows from 
the second part of \cite[Proposition 8.5]{MN},
since under the assumption of the strong Baum-Connes conjecture, the derived
crossed products $A\rtimes^{\mathbb L}G$ and $B\rtimes^{\mathbb L}G$
of \cite[Proposition 8.5]{MN} are $\KK$-equivalent to the full and reduced 
crossed products $A\rtimes_{(r)}G$ and $B\rtimes_{(r)}G$, respectively.
\end{proof}

The above proposition now implies:

\begin{prop}\label{prop-BC-fibration}
Suppose that $A=A(X)$ is a separable C*-algebra bundle over $X$
and let $\alpha:G\to \Aut(A)$ be a $C_0(X)$-linear
action of the second countable locally compact group
$G$ on $A(X)$. Assume that for each compact subgroup $K$ of $G$ 
the $C_0(X)$-algebra $A(X)\rtimes K$ is a $K_*$-fibration.
Then
\begin{enumerate}
\item If $G$ is exact and
satisfies BC for $f^*A$ for all continuous $f:\Delta^p\to X$, $p=0,1,2,..$ (in particular, if 
$G$ satisfies BC for all $G$-algebras $B$), then the reduced crossed product
$A(X)\rtimes_rG$ is a $K_*$-fibration.
\item If $G$ satisfies the strong Baum-Connes conjecture, 
then the full crossed product $A(X)\rtimes G$ is  a $K_*$-fibration.
If, in addition, $G$ is exact, the same is true for the reduced crossed product 
$A(X)\rtimes_r G$.
\item If $G$ satisfies the strong Baum-Connes conjecture and if $A(X)\rtimes K$ is a 
$\KK$-fibration for every compact subgroup 
$K\subseteq G$, then $A(X)\rtimes G$ is a $\KK$-fibration. If, in addition, $G$ is exact,
then $A(X)\rtimes_rG$ is a $\KK$-fibration, too.
\end{enumerate}
\end{prop}
\begin{proof}
If $G$ is exact, then $A\rtimes_rG$ is a $C_0(X)$-algebra with fibres $A_x\rtimes_rG$
and we have $f^*A\rtimes_rG\cong f^*(A\rtimes_rG)$ for all continuous $f:\Delta^p\to X$.
By the assumption on the compact subgroups of $G$ we see that the quotient map
$q_v:f^*A\to A_{f(v)}$  {induces an isomorphism 
$$
K_*(f^*A\rtimes K)\stackrel{\cong}{\longrightarrow} K_*(A_{f(v)}\rtimes K)
$$ 
for all compact subgroups} 
$K$ of $G$. Item (i) then follows from part (i) of Proposition \ref{prop-K}.

Similarly, (ii) and (iii) follow from parts (ii) and (iii) of Proposition \ref{prop-K} together with the fact that
the $C_0(X)$-algebra $A(X)\rtimes G$ has fibres $A_x\rtimes G$.
If $G$ is exact, the same argument works for $A(X)\rtimes_r G$.
\end{proof}

\begin{remark}\label{rem-gamma}
{\bf (1)} If $G$ has no  compact subgroups 
(e.g., $G=\RR^n$, $G=\ZZ^n$ or $G=F_n$, the free group with $n$ generators), 
then the requirement that  $A(X)\rtimes K$ being  {$K_*$-fibration}  (resp.  $\KK$-fibration)
in the above Proposition reduces to the requirement that $A(X)$ is a  {$K_*$-fibration}
(resp. $\KK$-fibration).
Thus, if any of the groups $G=\RR^n, \ZZ^n, F_n$ acts fibrewise
on a  {$K_*$-fibration}
(resp. $\KK$-fibration) $A(X)$, then $A(X)\rtimes_{(r)}G$ is also a 
{$K_*$-fibration}
(resp. $\KK$-fibration), since all of these groups are exact and satisfy the strong 
Baum-Connes conjecture. Of course, there are many other examples of such groups.

{\bf (2)} It follows from  \cite[Proposition 3.1]{CEN} that if $G$ is exact and 
has a $\gamma$-element in the sense of Kasparov \cite{Kas2}, 
and if $A(X)$ is a  continuous C*-algebra bundle over $X$, then 
$G$ satisfies $BC$ for $f^*A$ for all $f:\Delta^p\to X$
if (and only if) $G$ satisfies BC for $A_x$ for every fibre $A_x$ of $A$.
(The only if direction follows from taking the constant map $f:\Delta^p\to X; f(v)=x$
and using the fact that $f^*A=C(\Delta^p, A_x)$ is $\KK^G$-equivalent to $A_x$).
\end{remark}

%
%

If we specialize to continuous-trace algebras $A$ with base $X$, we can 
improve the results. For notation, we let $\K=\K(l^2(\NN))$ denote 
the compact operators on the infinite dimensional separable Hilbert space.
Recall that if $A(X)$ is any separable continuous-trace algebra with spectrum $X$,
then $A(X)\otimes \K$ is a locally trivial C*-algebra bundle with fibre $\K$.
Using this we get:

\begin{cor}\label{cor-conttrace}
Suppose that $G$ is a second countable locally compact group acting fibre-wise
on a separable continuous-trace $C^*$-algebra $A(X)$ with spectrum $X$. 
Then
\begin{enumerate}
\item If $G$ satisfies the strong Baum-Connes conjecture (e.g., if $G$ is a-$T$-menable), 
then $A(X)\rtimes G$ is a $\KK$-fibration. If, in addition, $G$ is exact, the same holds for
$A\rtimes_rG$.
\item If $G$ is exact and satisfies BC for $C(\Delta^p,\K)$ for all fibre-wise actions
on $C(\Delta^p,\K)$, $p\geq 0$, then $A\rtimes_rG$ is a $K_*$-fibration.
\item 
If $G$ is exact and  has a $\gamma$-element, and $G$ satisfies BC for $\K$, for 
all actions of $G$ on $\K$, then $A\rtimes_rG$ is a $K_*$-fibration.
\end{enumerate}
\end{cor}

Notice that by the results of \cite{CEN} condition (iii) is satisfied for all almost connected 
groups and for all linear algebraic groups over $\QQ_p$. 

\begin{proof}[Proof of Corollary \ref{cor-conttrace}]
The corollary will follow from  Proposition \ref{prop-BC-fibration} and Remark \ref{rem-gamma}
if we can show that $(f^*A\rtimes K)\otimes \K$ is a trivial $C(\Delta^p)$-algebra for all continuous 
maps $f:\Delta^p\to X$, since this will imply that $A\rtimes K$ is a $\KK$-fibration.

For this we first note that
$(f^*A\rtimes K)\otimes \K\cong (f^*A\otimes\K)\rtimes K$, where $K$ acts trivially on $\K$.
Using this we may simply assume
that $f^*A=C(\Delta^p,\K)$. 
But it follows then from \cite[Proposition 1.5]{ELPW} that any fibre-wise action
of a compact group $K$ on $C(\Delta^p,\K)=C(\Delta^p)\otimes \K$ is
exterior equivalent  to a diagonal action $\id\otimes\alpha^v$, with $\alpha^v$ the action on the fibre 
{$\K= C(\Delta^p,\K)_v$}.
Thus $C(\Delta^p,\K)\rtimes K$ is isomorphic to $C(\Delta^p, \K\rtimes_{\alpha^v}K)$
as bundles over $\Delta^p$.
\end{proof}

So far we only considered $K_*$- or $\KK$-fibrations, but we promised at the beginning of this section 
that we will provide also examples of $\RKK$-fibrations. Indeed, combining the above results 
with Corollary \ref{converse} gives 

\begin{cor}\label{cor-RKKfib}
Suppose that $A(X)$ is a separable nuclear and  locally trivial C*-algebra bundle 
and let $G$ be a second countable amenable group acting fibre-wise on $A(X)$.
Then the following are true
\begin{enumerate}
\item If $G$ has no compact subgroups then $A(X)\rtimes G$ is an $\RKK$-fibration.
\item If $A(X)$ is a continuous trace algebra with spectrum $X$, then 
$A(X)\rtimes G$ is an $\RKK$-fibration.
\end{enumerate}
\end{cor}
\begin{proof}
Since $G$ is amenable, it satisfies the strong Baum-Connes conjecture by \cite{HK}.
Moreover, a crossed product of a continuous C*-algebra bundle by a fibre-wise
group action of an amenable group $G$ is again a continuous C*-algebra bundle
by \cite{Will}. Since nuclearity is also preserved under taking crossed products by amenable groups,
it follows that $A(X)\rtimes G$ is a nuclear separable and continuous C*-algebra bundle.
Thus it follows from Corollary \ref{converse} that $A(X)\rtimes G$ is an $\RKK$-fibration if and only if it is a $\KK$-fibration. Hence the result follows 
from Remark \ref{rem-gamma} and Corollary \ref{cor-conttrace}.
\end{proof}

Of course, as an example of the above corollary we get a new proof of the 
fact that the non-commutative  {principal} torus bundles of \cite{ENO} are 
$\RKK$-fibrations, since, by definition,  they are crossed products of the form 
$C_0(X,\K)\rtimes \ZZ^n$ by $C_0(X)$-linear actions of $\ZZ^n$ on $C_0(X,\K)$.

\section{The group bundle corresponding to an $h$-fibration}\label{sec-bundle}

Suppose that $X$ is a locally compact space. By an (abelian)
{\em group bundle} $\mathcal G:=\{G_x: x\in X\}$ 
we understand a {functor from the homotopy groupoid of $X$ to
the category of (abelian) groups. It is given by} a family of groups $G_x$, $x\in X$, together with 
group isomorphisms  {$c_{\gamma}:G_x\to G_y$} 
for each continuous path
$\gamma: [0,1]\to X$ which starts at $x$ and ends at $y$, such that 
the following 
additional requirements are  satisfied:
\begin{enumerate}
\item If $\gamma$ and $\gamma'$ are homotopic paths from $x$ to $y$, then 
$c_\gamma=c_{\gamma'}$.
\item If $\gamma_1:[0,1]\to X$ and $\gamma_2:[0,1]\to X$ are paths from $x$ to $y$ 
and from $y$ to $z$, respectively, then 
{$$c_{\gamma_1\circ\gamma_2}=c_{\gamma_1}\circ c_{\gamma_2},$$}
where $\gamma_1\circ \gamma_2:[0,1]\to X$ is the usual composition of paths.
\end{enumerate}
It follows from the above requirements, that if $X$ is path connected, then all
groups $G_x$ are isomorphic and that we get a canonical action of the fundamental group 
$\pi_1(X)$ on each fibre $G_x$.

A morphism between two group bundles 
$\mathcal G=\{G_x: x\in X\}$ and $\mathcal G' =\{G'_x: x\in X\}$ is a family of group homomorphisms
$\phi_x:G_x\to G'_x$ which commutes with the maps $c_\gamma$. The trivial group bundle
is the bundle with every $G_x$ equal to a fixed group $G$ and all maps $c_{\gamma}$ 
being the identity. We then write $X\times G$ for this bundle.
If $X$ is path connected, then a given group bundle $\mathcal G$ on $X$ 
can be trivialized if and only if the action of $\pi_1(X)$ on the fibres $G_x$ are trivial.
In that case every path $\gamma$ from base points $x$ to $y$ induces the same morphism
{$c_{x,y}:G_x\to G_y$} and if we choose a fixed base point $x_0$,
the family of maps $\{c_{x,x_0}:x\in X\}$ is a group bundle 
isomorphism between the trivial group bundle $X\times G_{x_0}$  and the given 
bundle $\mathcal G=\{G_x:x\in X\}$.
It is now easy to check that every $h_*$-fibration $A(X)$ gives rise to a group bundle 
$\mathcal H_*:=\{h_*(A_x): x\in X\}$:

\begin{prop}\label{prop-bundle}
Suppose that $A(X)$ is an $h_*$-fibration. For any path $\gamma:[0,1]\to X$ 
with starting point $x$ and endpoint $y$ let
{$c_{\gamma}:h_*(A_x)\to h_*(A_y)$} denote the 
composition 
\begin{equation}\label{eq-cgamma}
\begin{CD}
h_*(A_y) @>\eps_{1,*}^{-1}>\cong> h_*(\gamma^*A) @>\eps_{0,*}>\cong> h_*(A_x)
\end{CD}
\end{equation}
Then $\mathcal H_*(A):=\{h_*(A_x): x\in X\}$ together with the above defined 
maps $c_{\gamma}$ is a group bundle over $X$.
A similar result holds for a cohomology theory $h^*$ if $A(X)$ is an $h^*$-fibration
(with arrows in (\ref{eq-cgamma}) reversed).
\end{prop}
\begin{proof}
It is clear that constant paths induce the identity maps
and that  {$c_{\gamma\circ \gamma'}=c_{\gamma}\circ c_{\gamma'}$ , where $\gamma\circ \gamma'$}
denotes composition of paths.
Moreover, if $\Gamma:[0,1]\times [0,1]\to X$ is a homotopy between
the paths $\gamma_0$ and $\gamma_1$ with equal starting and endpoints, then 
$c_{\gamma_0}$ and $c_{\gamma_1}$ both coincide with the 
composition $\eps_{(0,0),*}\circ \eps_{(1,1),*}^{-1}$, where $\eps_{(0,0)}$ and $\eps_{(1,1)}$
denote evaluations of $\Gamma^*A$ at the respective corners of $[0,1]^2$.
Hence we see that $c_\gamma$ only depends on the homotopy class of $\gamma$.
\end{proof}

\begin{defn}\label{defn-group-bundle}
Suppose that $h$ is a (co)homology theory on a good category $\C$ of $C^*$-algebras
and let $A(X)$ be an $h$-fibration. Then $\mathcal H_*:=\{h_*(A_x): x\in X\}$ (resp.
$\mathcal H^*(A):=\{h^*(A_x):x\in X\}$ if $h$ is a cohomolgy theory) together with 
the maps $c_{\gamma}:h_*(A_y)\to h_*(A_x)$  is called  the {\em $h_*$-group bundle} associated to 
$A(X)$. 
\end{defn}

\begin{remark}\label{rem-KK}
If $A(X)$ is a $\KK$-fibration, then it is in particular a $K_*$- and a $K^*$-fibration, where 
$K_*$ and  $K^*$ denote ordinary $K$-theory and $K$-homology. We shall denote the resulting 
group bundles by $\mathcal K_*(A)$ and $\mathcal K^*(A)$, respectively.
\end{remark}

\section{The Leray-Serre spectral sequence}\label{sec-lerayserre}

In this section we want to proof an analogue of the classical Leray-Serre spectral sequence for
topological Serre-fibrations. From the last remark of the previous section we know 
that if $A(X)$ is a $h$-fibration for a (co-)homology theory $h$, then we get the group bundle 
$\H_*(A)$ over $X$. It is  well-known in topology  that one can use such bundles
as coefficients for singular or simplicial (co-)homology on $X$. It is our aim to show
that every $h$-fibration over a finite dimensional simplicial complex $X$ admits a 
spectral sequence with $E_2$-terms isomorphic to the (co-)homology of $X$ with coefficient 
in $\H_*$ (resp. $\H^*$).

Assume that $X$ is a locally compact CW-complex and that $A$ is any $C_0(X)$-algebra.
For $p\geq 0$ let $X_p$ denote the $p$-skeleton of $X$ and we
set $X_p=\emptyset$ for a negative integer $p<0$. We will always assume that 
$X$ is {\bf finite dimensional} so that there exists a smallest integer $d$ (the dimension of $X$)
such that $X_p=X$ for all $p\geq d$.
For all $p$ we write $A_p:=A|_{X_p}$ and $A_{p,p-1}:=A|_{X_p\setminus X_{p-1}}$,
where we use $A|_\emptyset:=\{0\}$.
We then obtain short exact sequences 
$$0\to A_{p,p-1}\to A_p\to A_{p-1}\to 0.$$
If $h_*$ is any homology theory on a good subcategory $\C$ of $\Call$ such that all
algebras $A_p$ and $A_{p,p-1}$ are in $\C$, naturality of the  long exact sequences
\begin{equation}\label{eq-long}
\cdots \stackrel{\partial_{n+1}}{\longrightarrow} h_n(J)
\stackrel{i_*}{\longrightarrow} h_n(A)
\stackrel{q_*}{\longrightarrow} h_n(B)
\stackrel{\partial_{n}}{\longrightarrow} h_{n-1}(J)
\stackrel{i_*}{\longrightarrow} \cdots
\end{equation}
gives the 
following 
commutative diagram:
{\tiny $$
\begin{CD}
@. @Vq_*VV   @.  @Vq_*VV  @.     \\
@>>> h_{{q+1}}(A_{p+1}) @>\partial>> h_q(A_{p+2,p+1}) @>\iota_*>> h_q(A_{p+2}) @>\partial >> h_{{q-1}}(A_{p+3,p+2})  @>\iota_*>> \\
@. @Vq_*VV   @.  @Vq_*VV  @.      \\
@>>> h_{{q+1}}(A_{p}) @>\partial>> h_q(A_{p+1,p}) @>\iota_*>> h_q(A_{p+1}) @>\partial >> h_{{q-1}}(A_{p+2,p+1})  @>\iota_*>> \\
@. @Vq_*VV   @.  @Vq_*VV  @.      \\
@>>> h_{{q+1}}(A_{p-1}) @>\partial>> h_q(A_{p,p-1}) @>\iota_*>> h_q(A_{p}) @>\partial >> h_{{q-1}}(A_{p+1,p})  @>\iota_*>>\\
@. @Vq_*VV   @.  @Vq_*VV  @.     
\end{CD}
$$}
Here the vertical arrows are induced by the quotient maps $q:A_p\to A_{p-1}$,
the maps $\iota^*: h_q(A_{p,p-1}) \to h_q(A_p) $ are induced by the inclusions
$\iota: A_{p,p-1}\to A_p$ and the maps $\partial: h_{{q+1}}(A_{p-1})\to  h_q(A_{p,p-1})$
denote the boundary maps in the long exact sequence (\ref{eq-long}).
Hence, the upper staircase of this diagram
forms the sequence (\ref{eq-long}).
Now writing $H^{p,q}:= h_q(A_p)$, {$E_1^{p,q}:=h_{q}(A_{p,p-1})$},
$\mathcal H:=\oplus_{p,q} H^{p,q}$ and $\mathcal E:=\oplus E_1^{p,q}$ 
we obtain an exact couple
$$
\xymatrix{
  \mathcal H\ar@{->}[rr]^{q^*}  & & \mathcal H \\
  & \ar@{->}[ul]^{\iota^*}  \mathcal E   \ar@{<-}[ur]_{\partial} & 
  }
$$
from which we obtain by the general procedure (which, for example,  is explained in 
\cite{McC}) a spectral sequence $\{E_r^{p,q}, d_r\}$ with $E_{\infty}$-terms
$E_{\infty}^{p,q}=F_p^q/F_{p+1}^q$
 with {$F_p^q:=\ker\big(h_{q}(A)\to h_q(A_p)\big)$.}
Since $F_{p}^q=h_q(A)$ for $p<0$ and $F_p^q=\{0\}$ for $p\geq d$, the dimension of $X$,
it follows that the spectral sequence converges to $h_q(A)$.
This means that we obtain a filtration
$$\{0\}=F_{d}^q\subseteq F_{d-1}^q\subseteq\cdots\subseteq F_{-1}^q=h_{q}(A)$$
of subgroups $F_p^q$ of $h_q(A)$ such that the sub-quotients can be computed (at least in principle)
by our  spectral sequence.

Similarly, if we start with a cohomology theory $h^*$ on $\C$, we consider the diagram
{\tiny $$
\begin{CD}
@. @Vq^*VV   @.  @Vq^*VV  @.      \\
@<<< h^{q-1}(A_{p-1}) @<\partial<< h^q(A_{p,p-1}) @<\iota^*<< h^q(A_p) @<\partial <<  h^{q+1}(A_{p+1,p})  @<\iota^*<< \\
@. @Vq^*VV   @.  @Vq^*VV  @.       \\
@<<< h^{q-1}(A_{p}) @<\partial<< h^q(A_{p+1,p}) @<\iota^*<< h^q(A_{p+1}) @<\partial <<  h^{q+1}(A_{p+2,p+1})  @<\iota^*<< \\
@. @Vq^*VV   @.  @Vq^*VV  @.       \\
@<<< h^{q-1}(A_{p+1}) @<\partial<< h^q(A_{p+2,p+3}) @<\iota^*<< h^q(A_{p+2}) @<\partial <<  h^{q+1}(A_{p+31,p+2})  @<\iota^*<<\\
@. @Vq^*VV   @.  @Vq^*VV  @.     
\end{CD}
$$}
which provides a spectral sequence $\{E^r_{p,q}, d^r\}$ with $E^{\infty}$-terms
{$E^{\infty}_{p,q}:=F^p_q/F^{p-1}_q$
where $F^p_q:=\im\big(h^{q}(A_p)\to h^{q}(A)\big)$.} Again, since $X$ is finite dimensional,
the spectral sequence converges to $h^q(A)$. 
Hence, at this stage we arrive at

\begin{prop}\label{prop-exact1}
Suppose that $X$ is a finite dimensional CW-complex and let $A(X)$ be a C*-algebra
bundle over $X$.
Suppose that $\C$ is  a good subcategory of $\Call$ so that $A_p:=A(X_p)\in \C$ 
for every $p$-skeleton of $X$. Then, if $h_*$ is a homology theory (resp. $h^*$ is a cohomology 
theory) on $X$ there exists a spectral sequence $\{E_r^{p,q}, d_r\}$ (resp. $\{E^r_{p,q}, d^r\}$)
which converges to $h_*(A)$ (resp. $h^*(A)$) as described above.
\end{prop}

\begin{remark}\label{rem-E2}
Let us denote by $\{U^p_i: i\in I_p\}$ the open $p$-cells of $X$. We then have 
$$A_{p,p-1}=\oplus_{i\in I_p}A(U^p_i).$$
If $X$ is a finite simplicial complex this sum is finite and it follows from additivity of $h_*$
(resp. $h^*$) that 
{$$E_1^{p,q}=\oplus_{i\in I_p}h_{q}(A(U^p_i))\quad\quad\text{(resp.}\; 
E^1_{p,q}=\oplus_{i\in I_p}h^q(A(U^p_i))\;\text{)}.$$}

Of course, if $h_*$ (resp. $h^*$) is $\sigma$-additive or $\sigma$-multiplicative
we get similar 
infinite direct sum or product decompositions in case where $X$ is a $\sigma$-finite (i.e., $X$ has countably many cells). In any case we shall assume that $X$ is locally finite.
The $d_1$-differential is then determined by the maps
$$d_{1,q}^{(p,i),(p+1,j)}:h_q(A(U^p_i))\to {h_{q-1}}(A(U^{p+1}_j))$$
given by the composition 
$$
\begin{CD}
h_q(A(U^p_i)) @>I_q^{p,i}>> h_q(A_{p,p-1})
@>i_*>> h_{q}(A_p)\\
@>\partial >> h_{{q-1}}(A_{p+1,p})
@>Q_{{q-1}}^{p+1,j}>> h_{{q-1}}(A(U^{p+1}_j))
\end{CD}
$$
where $I^{p,i}:A(U^p_i)\to A_{p,p-1}$ denotes the inclusion and $Q^{p,i}:A_{p,p-1}\to A(U^p_i)$
denotes the quotient map. Similarly, for a cohomology theory $h^*$ we get maps
$$d^{1,q}_{(p,i),(p-1,j)}:h^q(A(U^p_i))\to h^{{q+1}}(A(U^{p-1}_j))$$
which are given by the 
compositions 
$$
\begin{CD}
h^q(A(U^p_i))@>Q^q_{p,i}>> h^q(A_{p,p-1}) @>\partial>> h^{{q+1}}(A_{p-1}) \\
@>i^*>>
 h^{{{q+1}}}(A_{p-1,p-2}) @>I^{{{q+1}}}_{p-1,j}>> h^{{{q+1}}}(A(U^{p-1}_j)).
 \end{CD}
$$
with similar meanings for $Q_{p,i}$ and $I_{p,i}$. It follows then that $E_2^{p,q}$ (resp. $E^2_{p,q}$)
is the cohomology  (resp. homology) of the complex build out of the 
above given data.
\end{remark}

We now want to study the groups $h_*(A(U^p_i))$ and the maps $d_{1,q}^{(p,i),(p+1,j)}$ 
more closely  in case where $A(X)$ is an $h_*$-fibration  and $X$ is a (finite) simplicial complex.
In particular, we want to give a better computation of the $E_2$-terms.
We shall restrict to the case of a homology theory 
on a good category
$\C$ of C*-algebras throughout, noting that similar 
arguments work for a cohomology theory as well. 

We start with introducing some notation:
As before, we let 
$$\Delta^p:=<v_0,\ldots, v_p>$$ denote the oriented closed $n$-simplex
with vertices $v_0,\ldots, v_p$, 
we let $\iDelta^p$ denote
its interior and we  let $\bDelta^p$ denote its boundary. 
If $0\leq k\leq p$ we shall consider $\Delta^k=<v_0,\ldots, v_k>$ as a subset of $\Delta^p$ 
and for $i\in\{0,\ldots, n\}$ we 
write  $\Delta^{p-1}_i:=<v_0,\ldots, v_{i-1},v_{i+1},\ldots v_p>$ for the oriented $i$th face of $\Delta^p$.
If $X$ is our given simplicial complex we write $\{\Delta^p(j):j\in I_p\}$ for the set of closed 
$p$-simplexes in $X$, and we let
$$\sigma^p(j):\Delta^p\to \Delta^p(j)\subseteq X$$
denote an explicit affine homeomorphism between the standard simplex $\Delta^p$ and $\Delta^p(j)$.

To study the differential {$d:E^1_{p,q}\to E^2_{p+1,q-1}$} in the above remark, we first need 
to study the simple case where $X=\Delta^p$ itself.  Recall that if $A=A(X)$ is an 
$h_*$-fibration with $X$ simply connected, then for each $x,y\in X$ there are 
unique isomorphisms $\Phi_{x,y}:h_*(A_x)\to h_*(A_y)$, which, for any chosen 
path $\gamma:[0,1]\to X$ with $\gamma(0)=y, \gamma(1)=x$, satisfy the equations
$$\Phi_{x,y}\circ \ev_{1,*}=\ev_{0,*}:h_*(\gamma^*A[0,1])\to h_*(A_y).$$

\begin{lem}\label{lem-path}
Suppose that $A=A(\Delta^P)$ is an $h_*$-fibration (resp. $h^*$-fibration) with $p\geq 1$.
Then for every $x\in \Delta^p$ evaluation at $x$ induces an isomorphism
$$\ev_{x,*}:h_*(A(\Delta^p))\to h_*(A_x).$$
Moreover, if $y$ is any other point in $\Delta^p$, then $\ev_{y,*}=\Phi_{x,y}\circ \ev_{x,*}$
(and similar statements for $h^*$-fibrations).
\end{lem}
\begin{proof} The first statement holds by definition of an $h_*$-fibration.
So we only have to check that $\ev_{y,*}=\Phi_{x,y}\circ \ev_{x,*}$
for any pair $x,y\in \Delta^p$.
Let $\gamma:[0,1]\to \Delta^p$ denote any path connecting 
$y=\gamma(0)$ with $x=\gamma(1)$. Since $\gamma$ is a proper map, 
\cite[Lemma 1.3]{ENO}
provides a  $*$-homomorphism $\Phi_{\gamma}:A(\Delta^p)\to \gamma^*A([0,1])$ and 
it is clear from the construction of $\Phi_{\gamma}$ that $\ev_x=\ev_1\circ \Phi_{\gamma}$
and $\ev_y=\ev_0\circ \Phi_{\gamma}$. The result now follows from the definition of $\Phi_{x,y}$.
\end{proof}

\begin{lem}\label{lem-contact}
Let $1\leq p$ and suppose that the $C(\Delta^p)$-algebra $A$ is an $h_*$-fibration.
Let $W\subseteq \Delta^p$ be any  set which is obtained from $\Delta^p$  by
removing a union of $k$ faces of dimension $p-1$ from $\Delta^p$
with $1\leq k\leq p$. Then $h_*(A(W))=0$.
\end{lem}
\begin{proof} 
The proof is by induction on $p$ and $k$. If $p=1$ then $k=1$ and $W$
is homeomorphic to $[0,1)$. Since evaluation $A([0,1])\to A_1$ induces an isomorphism of 
$h_*$-groups, it follows from the long exact sequence corresponding to
$0\to A([0,1))\to A([0,1])\to A_1\to 0$ that $h_*(A(W))=0$.

Suppose now that $p>1$.
If $\Delta^{p-1}_i$ is any closed face of $\Delta^p$, then it follows from 
the properties of $h_*$-fibrations that the quotient map $q_i:A(\Delta^p)\to A(\Delta^{p-1}_i)$
induces an isomorphism $h_n(A(\Delta^p))\stackrel{q_{i,*}}{\longrightarrow} h_n(A(\Delta^{p-1}_i))$,
since composition with evaluation at any vertex $v$ of $\Delta^{p-1}_i$ induces 
isomorphisms $h_n(A(\Delta^p))\cong h_n(A(\Delta^{p-1}_i))\cong h_n(A_v)$. Hence, if $W=\Delta^p\smallsetminus \Delta^{p-1}_i$,
the long exact sequence of $h_*$-groups corresponding to the short exact sequence 
$0\to A(W)\to A(\Delta^p)\to A(\Delta^{p-1}_i)\to 0$ implies that  $h_*(A(W))=0$.

Suppose now that $k>1$ and let $F:= \Delta^p\smallsetminus W$.
Then we can write $F$ as a union  $F'\cup\Delta^{p-1}_i$,
where $F'$ is a union of $k-1$ faces. By the induction assumption we know that
$h_*(A(W'))=0$ for $W':=\Delta^p\smallsetminus F'$. Moreover, $W'\smallsetminus W$ 
is equal to $W'':=\Delta^{p-1}_i\setminus F''$, where $F''$ is some
union of $l$ $p-2$-dimensional faces with $1\leq l\leq p-1$. Hence, by induction 
assumption, we have $h_*(A(W''))=0$ and then all terms in the exact sequence
$h_q(A(W'))\to h_q(A(W))\to h_q(A(W''))$ must be zero.
\end{proof}

\begin{defremark}\label{defrem-canonical}
For the $p$-simplex $\Delta^p=<v_0, \ldots, v_p>$ let $W_i:=\iDelta^p\cup\iDelta^{p-1}_i$.
If the $C(\Delta^p)$-algebra $A$ is an $h_*$-fibration, then it follows from the above lemma that 
$h_*(A(W_i))=0$, which then 
implies that the boundary map {$\partial_i: h_q(A(\iDelta^{p-1}_i))\to h_{{{q-1}}}(A(\iDelta^p))$}
corresponding to the short exact sequence
$0\to A(\iDelta^p)\to A(W_i)\to A(\iDelta^{p-1}_i)\to 0$
is an isomorphism for all $q$.

In particular, there is a chain of isomorphisms
{$$h_q(A(\Delta^p))\stackrel{\ev_{v_0,*}}{\longrightarrow}
h_q(A_{v_0})\stackrel{\partial_1}{\longrightarrow} h_{{{q-1}}}(A(\iDelta^1))\stackrel{\partial_2}{\longrightarrow}
\cdots \stackrel{\partial_p}{\longrightarrow}h_{{{q-p}}}(A(\iDelta^p)),$$}
which we shall call the {\em canonical oriented isomorphism 
$$\Phi_q^p:h_q(A(\Delta^p))\stackrel{\cong}{\longrightarrow}
h_{q-p}(A(\iDelta^p)).$$}
\end{defremark}

It is important for us to understand how the canonical isomorphisms depends on the orientation of the simplex $\Delta^p$. We start this investigation with two basic observations. The  first
considers the case $\Delta^1=[0,1]$:

\begin{lem}\label{lem-orient1}
Suppose that $A([0,1])$ is a C*-algebra bundle which is an $h_*$-fibration.
Let 
$$\partial_0:h_q(A_0)\to h_{q-1}(A(0,1))\quad\text{and}\quad 
\partial_1:h_q(A_1)\to h_{q-1}(A(0,1))$$
denote the isomorphisms given by the connecting maps in  the long exact sequences 
related to evaluation of $A([0,1))$ and $A((0,1])$ at $0$ and $1$, respectively.
Then $\partial_0=-\, \partial_1\circ \Phi_{0,1}$. 
\end{lem}
\begin{proof} Consider the long exact sequence 
$$\cdots\to h_q(A(0,1))\to h_q(A[0,1])\to h_q(A_0\oplus A_1)\stackrel{\partial}{\to} {h_{q-1}(A(0,1))}\to\cdots$$
corresponding to
$$0\to A((0,1))\to A([0,1])\to A_0 {\oplus} A_1\to 0.$$
The connecting map in this sequence equals $\partial_0+\partial_1$, which follows from 
naturality of the long exact sequence together with the diagram
$$
\begin{CD}
0 @>>>   A((0,1)) @>>> A([0,1)) @>>> A_0 @>>>  0\\
@.         @V=VV     @V\iota VV         @VV\iota V\\
0 @>>>   A((0,1)) @>>>{A([0,1])} @>>> A_0\oplus A_1 @>>>  0\\
  @.       @A=AA     @A\iota AA         @AA\iota A\\
0 @>>>   A((0,1)) @>>> A((0,1]) @>>> A_1 @>>>  0.
\end{CD}
$$
Exactness
then gives $$0=\partial_0\circ \ev_{0,*}+\partial_1\circ \ev_{1,*}$$ where
$\ev_{i,*}:h_q(A([0,1]))\to h_q(A_i)$, $i=0,1$, denotes the evaluation isomorphism.
Thus, we get
$ \partial_0\circ \ev_{0,*}=-\, \partial_1\circ \ev_{1,*}$ and composing both sides with
$\ev_{0,*}^{-1}$ on the right gives the lemma.
\end{proof}

In the next lemma we compare 
the isomorphisms  
$$h_q(A(\iDelta^{p-1}_{p-1}))\cong {h_{q-1}}(A(\iDelta^p))$$
and 
$$h_q(A(\iDelta^{p-1}_{p}))\cong  {h_{q-1}}(A(\iDelta^p))$$
as defined in 
\ref{defrem-canonical}  for the $p$-th and the $(p-1)$-st
face of $\Delta^p$:

\begin{lem}\label{lem-orient2}
Suppose that $p>1$. Then the compositions
$$h_r(A(\iDelta^{p-2}))\cong h_{r-1}(A(\iDelta^{p-1}_p))\cong h_{r-2}(A(\iDelta^p))$$
and
$$
h_r(A(\iDelta^{p-2}))\cong h_{r-1}(A(\iDelta^{p-1}_{p-1}))\cong h_{r-2}(A(\iDelta^p))$$
differ by the factor $-1$. 
\end{lem}
\begin{proof}
Let 
\begin{align*}
&W=\iDelta^{p-2}\cup \iDelta^{p-1}_p\cup\iDelta^{p-1}_{p-1}\cup \iDelta^p,\quad\quad\;\;
W_0=\iDelta^{p-2}\cup \iDelta^{p-1}_p,\\
& W_1=\iDelta^{p-2}\cup \iDelta^{p-1}_{p-1},\quad \quad \quad \quad \quad\quad \quad \quad
W_{0,1}=\iDelta^{p-1}_p\cup\iDelta^{p-1}_{p-1}\cup \iDelta^{p-2}.
\end{align*}
Then it follows from Lemma \ref{lem-contact} that all groups 
\begin{align*}
h_*(A(W))&=h_*(A(W_1))=h^*(A(W_0))\\
&=h_*(\iDelta^p\cup\iDelta^{p-1}_p)
=h_*(\iDelta^p\cup\iDelta^{p-1}_{p-1})
\end{align*}
vanish.
Since $h^*(A(W))=0$, the boundary map in the long exact sequence 
for
$$0\to A(\iDelta^p)\to A(W)\to A(W_{0,1})\to 0$$
induces an isomorphism 
$$\partial_{0,1}: h_{l}(A(W_{0,1}))\to {h_{l-1}}(A(\iDelta^p)).$$
Naturality of the long exact sequences and the diagram
$$
\begin{CD}
0 @>>>   A(\iDelta^p) @>>> A(\iDelta^p\cup\iDelta^{p-1}_p) @>>> A(\iDelta^{p-1}_p) @>>>  0\\
 @.        @V=VV     @V\iota VV         @VV\iota V\\
0 @>>>   A(\iDelta^p) @>>> A(W) @>>>  A(W_{0,1}) @>>>  0\\
 @.        @A=AA     @A\iota AA         @AA\iota A\\
0 @>>>   A(\iDelta^p) @>>> A(\iDelta^p\cup\iDelta^{p-1}_{p-1}) @>>> A(\iDelta^{p-1}_{p-1}) @>>>  0
\end{CD}
$$
together with the Five-Lemma shows that the inclusions of 
$\iDelta^{p-1}_p$ and $\iDelta^{p-1}_{p-1}$ into $W_{0,1}$ induce isomorphisms 
$$\iota_p: h_l(A(\iDelta^{p-1}_p))\stackrel{\cong}{\to} { h_{l}}(A(W_{0,1}))$$ and 
$$\iota_{p-1}:h_l(A(\iDelta^{p-1}_{p-1}))\stackrel{\cong}{\to} { h_{l}}(A(W_{0,1}))$$
such that 
\begin{equation}
\label{eq-part}
\partial_{0,1}\circ \iota_p= \partial_p\quad\text{and}\quad
\partial_{0,1}\circ \iota_{p-1}= \partial_{p-1},
\end{equation}
where, as before, $\partial_i:h_l(A(\iDelta^p_i))\to {h_{l-1}}(A(\iDelta^p))$ denotes the 
isomorphism for the $i$th face of $\Delta^p$.
We now look at the diagram
\begin{equation}\label{commdia}
\begin{CD}
0 @>>>   A(\iDelta^{p-1}_p) @>>> A(W_0) @>>> A(\iDelta^{p-2}) @>>>  0\\
   @.      @A q AA     @A q AA        @AA = A\\
0 @>>>   A(\iDelta^{p-1}_p\cup\iDelta^{p-1}_{p-1}) 
@>>> A(W_{0,1}) @>>> A(\iDelta^{p-2}) @>>>  0\\
  @.       @V q VV    @V q VV         @VV = V\\
0 @>>>   A(\iDelta^{p-1}_{p-1}) @>>> A(W_1) @>>> A(\iDelta^{p-2}) @>>>  0
\end{CD}
\end{equation}
It implies that the composition of the boundary map
$$h_{r}(A(\iDelta^{p-2}))\to {h_{r-1}}(A(\iDelta^{p-1}_p\cup\iDelta^{p-1}_{p-1}) )$$
followed by the projections onto ${h_{r-1}}(A(\iDelta^{p-1}_p))$ and $
{h_{r-1}}(A(\iDelta^{p-1}_{p-1}))$,
respectively, coincide with the boundary maps
$$\partial^{p-2}_p:h_{r}(A(\iDelta^{p-2}))\cong{h_{r-1}}(A(\iDelta^{p-1}_p))$$
and 
$$\partial^{p-2}_{p-1}:h_{r}(A(\iDelta^{p-2}))\cong{h_{r-1}}(A(\iDelta^{p-1}_{p-1})),$$
respectively. On the other hand, it is clear that the isomorphisms
$$\iota_p:{h_{r-1}}(A(\iDelta^{p-1}_p))\cong{h_{r-1}}(A(W_{0,1}))$$ and
$$\iota_{p-1}:{h_{r-1}}(A(\iDelta^{p-1}_{p-1}))\cong{h_{r-1}}(A(W_{0,1}))$$
factorise via the inclusions of 
${h_{r-1}}(A(\iDelta^{p-1}_p))$ and ${h_{r-1}}(A(\iDelta^{p-1}_{p-1}))$
as direct summands of ${h_{r-1}}(A(\iDelta^{p-1}_p\cup\iDelta^{p-1}_{p-1}) )$, and 
that these inclusions invert the projections on the summands which are induced from 
the quotient maps in the first vertical row of  Diagram (\ref{commdia}). This implies exactness
of the sequence
{\small $$
\begin{CD}
h_r(A(\iDelta^{p-2})) @>\partial^{p-2}_{p}\oplus\partial^{p-2}_{p-1}>>
{h_{r-1}}(A(\iDelta^{p-1}_p))\oplus {h_{r-1}}(A(\iDelta^{p-1}_{p-1}))
@>\iota_p+\iota_{p-1} >>{h_{r-1}}(A(W_{0,1})).
\end{CD}
$$}
Combining this with Equation (\ref{eq-part}) gives 
$$0=\partial_{0,1}\circ (i_p\circ \partial^{p-2}_p +i_{p-1}\circ \partial^{p-2}_{p-1})=
\partial_p\circ \partial^{p-2}_p+\partial_{p-1}\circ \partial^{p-2}_{p-1},$$
which finally finishes the proof.
\end{proof}

We now want to consider an arbitrary permutation  
$\varphi:\{0,\ldots,p\}\to\{0,\ldots,p\}$. We shall denote by the same 
letter the unique affine isomorphism $\Delta^p\stackrel{\cong}{\longrightarrow}
\Delta^p$ which is induced by applying $\varphi$ on the vertices. Let 
$$\Psi_{\varphi}:A(\Delta^p)\to \varphi^*A(\Delta^p)$$
be the isomorphism of \cite[Lemma 1.3]{ENO}. It clearly restricts to an isomorphism,
also denoted $\Psi_{\varphi}$, between the ideals $A(\iDelta^p)$ and $\varphi^*A(\iDelta^p)$.
We then get

\begin{prop}\label{prop-affine}
Let $\sign(\varphi)$ denote the sign of the permutation $\varphi:\{0,\ldots, p\}\to \{0,\ldots, p\}$.
Then the following diagram commutes
\begin{equation}\label{eq-affine}
\begin{CD}
h_q(A(\Delta^p))     @>\Psi_{\varphi,*}>>      h_q(\varphi^*A(\Delta^p))\\
@V \Phi_p^q VV         @VV\sign(\varphi) \Phi_p^q V\\
h_{q+p}(A(\iDelta^p))    @>\Psi_{\varphi,*}>>   h_{q+p}(\varphi^*A(\iDelta^p)).
\end{CD}
\end{equation}
\end{prop}
\begin{proof}
Since every permutation is a product of transpositions which interchange two 
neighbours in $\{0,\ldots,p\}$, we may assume without loss of generality that 
$\varphi$ interchanges $l$ with $l+1$.

If we identify $A(\Delta^p)$ with $\varphi^*A(\Delta^p)$ via $\Psi_{\varphi}$, 
and if we write $\Delta^l_{\varphi}$ for the $l$-dimensional face 
$<\varphi(v_0),\ldots, \varphi(v_l)>\subseteq \Delta^p$,
the  above diagram restricts to showing that the isomorphism given by the composition
\begin{align*}
h_q(A(\Delta^p))&\stackrel{\ev_{v_{\varphi(0)},*}}{\longrightarrow}
h_q(A_{\varphi(v_0)})
\longrightarrow h_{q+1}(A(\iDelta^1_{\varphi}))\\
&\longrightarrow h_{q+2}(A(\iDelta^2_{\varphi}))
\longrightarrow
\cdots
\longrightarrow h_{q+p}(A(\iDelta^p))
\end{align*}
differs from the canonical isomorphism $\Phi_q^p:h_q(A(\Delta^p))\to h_{q+p}(A(\iDelta^p))$
by  $\sign(\varphi)$. For this we first remark, that by Lemma \ref{lem-path}
we have the equality 
$$\Phi_{v_0,\varphi(v_0)}\circ \ev_{v_0,*}=\ev_{\varphi(v_0),*}:h_q(A(\Delta^p))\to h_q(A_{\varphi(v_0)}),
$$
which then implies that we have to prove that the isomorphisms
$$\Theta:h_q(A_{v_0})\cong h_{q-1}(A(\iDelta^1))\cong h_{q-2}(A(\iDelta^2))\cdots\cong h_{q-p}(A(\iDelta^p))$$
and
{$$\Theta_{\varphi}:h_q(A_{\varphi(v_0)})\cong h_{q-1}(A(\iDelta^1_{\varphi}))
\cong h_{q-2}(A(\iDelta^2_{\varphi}))\cdots\cong h_{q-p}(A(\iDelta^p))$$}
are related via 
$$\Theta=\sign(\varphi)\Theta_{\varphi}\circ \Phi_{v_0,\varphi(v_0)}.$$

If $\varphi$ permutes $0$ with $1$, then we have $\Delta^j_{\varphi}=\Delta^j$ for all $j\geq 1$, so the 
above equation reduces to the case $p=1$. This case is taken care for by Lemma \ref{lem-orient1}
above. If $\varphi$ permutes $l$ with $l+1$ for some  $l>0$, then we have 
$\Delta^j_{\varphi}=\Delta^j$ for all $j\geq l+1$, so we may assume
without loss of generality that $l=p-1$. We then also have $\Delta^{p-2}_\varphi=\Delta^{p-2}$
and all we have to show is that the compositions
$$h_r(A(\iDelta^{p-2}))\cong  {h_{r-1}}(A(\iDelta^{p-1}_p))\cong  {h_{r-2}}(\iDelta^p))$$
and 
$$h_r(A(\iDelta^{p-2}))\cong  {h_{r-1}}(A(\iDelta^{p-1}_{p-1}))\cong  {h_{r-2}}(\iDelta^p))$$
differ by the factor $-1$.
But this is shown in Lemma \ref{lem-orient2}. This completes  the proof.
\end{proof}

In order to state the following important corollary, let us note that if $A(\Delta^p)$ is an 
$h_*$-fibration and if $\Delta^{p-1}_l$ is a face of $\Delta^p$, then the quotient map
$A(\Delta^p)\to A(\Delta^{p-1}_l)$ induces an isomorphism
\begin{equation}\label{eq-res}
\res_l:h_*(A(\Delta^p))\to h_*(A(\Delta^{{p-1}}_l))
\end{equation}
which follows from the simple fact that evaluation at any common vertex induces an
isomorphism in $h_*$-theory for both algebras. With this notation we now get

\begin{cor}\label{cor-lface}
Suppose that $A(\Delta^p)$ is an $h_*$-fibration and let
$$\partial_l:h_{q-p+1}(A(\iDelta^{p-1}_l))\to h_{q-p}(A(\iDelta^p))$$ denote 
the isomorphism of Definition \ref{defrem-canonical}. Let
$\Phi_p^q:h_q(A(\Delta^p))\to h_{q-p}(A(\iDelta^p))$ denote the canonical isomorphism 
and let
$$\Phi_{p-1}^q:h_q(A(\Delta^{p-1}_l))\to h_{q-p+1}(A(\iDelta^{p-1}_l))$$
be the canonical isomorphism  for $A(\Delta^{p-1}_l)$, with respect to the
orientation  of $\Delta^{p-1}_l$ inherited from $\Delta^p$.
Then 
$$\partial_l\circ \Phi_{p-1}^q\circ \res_l= (-1)^{{p-l}} \Phi_p^q.$$
\end{cor}
\begin{proof}
Consider the permutation $\varphi:\{0,\ldots,p\}\to\{0,\ldots,p\}$
defined by  
\begin{itemize}
\item $\varphi(k)=k$ for $k=0,\ldots,l-1$;
\item $\varphi(k)=k+1$ for $k=l,\ldots,p-1$;
\item $\varphi(p)=l$. 
\end{itemize}
Then $\sign(\varphi)=(-1)^{{p-l}}$. If we identify $\varphi^*A$ with $A$ 
via $\Phi_{\varphi}$, as done in the proof of Proposition \ref{prop-affine}, then 
the isomorphism $\Phi_p^q$ for the fibration $\varphi^*A(\Delta^p)$
just becomes $\partial_l\circ \Phi_{p-1}^q$. By Proposition \ref{prop-affine}
this differs from the isomorphism $\Phi_p^q$ for $A(\Delta^p)$ by the factor
$\sign(\varphi)=(-1)^{{p-l}}$.
\end{proof}

\noindent
{\bf The Leray-Serre spectral theorem.}
We are now going back to the situation of Proposition \ref{prop-exact1} in the special case
where $X$ is a finite dimensional simplicial complex.  In what follows, we write
$C_p$ for the set of oriented 
closed $p$-simplexes in $X$. To be more precise,  we consider any element
$\sigma\in C_p$ as a given affine realization $\sigma:\Delta^p\to \Delta^p_{\sigma}\subseteq X$ 
of the closed $p$-cell $\Delta^p_{\sigma}$ of $X$, which then induces an orientation on $\Delta^p_{\sigma}$.
If $\sigma\in C_p$, we write $\sigma_l:\Delta^{p-1}_l\to \Delta^{p-1}_{\sigma,l}\subseteq X$ 
for the $l$-th face of $\sigma$. Then there exists 
a unique element $\tau\in C_{p-1}$ such that $\tau(\Delta^{p-1})=\sigma_l(\Delta^{p-1}_l)$
and then a unique affine transformation $\varphi_{\sigma,\tau}^l:\Delta^{p-1}\to \Delta^{p-1}{\cong}\Delta^{p-1}_l$
such that 
\begin{equation}\label{eq-permute}
\tau=\sigma_l\circ \varphi_{\sigma,\tau}^l.
\end{equation}
These give precisely the gluing data for our simplicial complex $X$.

As outlined in Remark \ref{rem-E2},
under suitable finiteness conditions explained there, the $E_1$ terms are then given by
$$E_1^{p,q}=\oplus_{\sigma\in C_p}h_q(A(\iDelta^p_{\sigma}))$$
and the differential $d:E_1^{p-1,{q+1}}\to E_1^{p,q}$ applied to the direct summand 
$h_{{q+1}}(A(\iDelta^{p-1}_\tau))$ of $E_1^{p-1,{q+1}}$ projects to  the direct summand
$h_q(A(\iDelta^p_\sigma))$ of $E_1^{p,q}$
via the chain of maps
\begin{equation}\label{eq-tau-sigma}
\begin{split}
h_{{{q+1}}}(A(\iDelta^{p-1}_\tau))&\stackrel{I_{{{q+1}}}^{p-1}(\tau)}{\longrightarrow}
h_{{{q+1}}}(A_{p-1,p-2})\\
&\stackrel{\iota_*}{\longrightarrow}
h_{{{q+1}}}(A_{p-1})\stackrel{\partial}{\longrightarrow} 
h_q(A_{p,p-1})\stackrel{Q_q^p(\sigma)}{\longrightarrow}
h_q(A(\iDelta^p_{\sigma})),
\end{split}
\end{equation}
where as before, $A_p=A(X_p)\,,
A_{p,p-1}=A(X_p\setminus X_{p-1})\cong\oplus_{\sigma\in
 C_p}A(\iDelta^p_{\sigma})$,
{$I^{p-1}(\tau):A(\iDelta^{p-1}_\tau)\to A_{p-1,p-2}$ is the
inclusion map corresponding to the open cell $\iDelta^p_{\sigma}$ and
$Q^p(\sigma):A_{p,p-1}\to A(\iDelta^p_{\sigma})$ is the quotient map.}
We now apply the inverses of the canonical isomorphisms
$${\Phi_{{p+q}}^p: h_{p+q}(A(\Delta^p_{\sigma}))} \to h_q(A(\iDelta^p_{\sigma}))$$ 
(see Definition \ref{defrem-canonical})
to each simplex 
$\sigma\in C_p$ (and similarly for $\tau\in C_{p-1}$) which gives us isomorphisms
$$E_1^{p,q}\cong \oplus_{\sigma\in C_p}h_{{p+q}}(A(\Delta^p_{\sigma}))\quad
\text{and}\quad
E_1^{p-1,{{q+1}}}\cong \oplus_{\tau\in C_{p-1}}h_{{p+q}}(A(\Delta^{p-1}_{\tau})).$$
In this picture, the differential is described on the summands via
\begin{equation}\label{eq-differential}
\begin{split}
d_{\tau}^{\sigma} : h_{{p+q}}(A(\Delta^{p-1}_{\tau}))&\stackrel{\Phi_{{p+q}}^{p-1}}{\longrightarrow}
h_{{{q+1}}}(A(\iDelta^{p-1}_\tau))\stackrel{I_{{{q+1}}}^{p-1}(\tau)}{\longrightarrow}
h_{{{q+1}}}(A_{p-1,p-2})
\stackrel{\iota_*}{\longrightarrow}
h_{{{q+1}}}(A_{p-1})\\
&\stackrel{\partial}{\longrightarrow} 
h_q(A_{p,p-1})\stackrel{Q_q^p(\sigma)}{\longrightarrow}
h_q(A(\iDelta^p_{\sigma}))
\stackrel{(\Phi_{{q+p}}^p)^{-1}}{\longrightarrow}
h_{{p+q}}(A(\Delta^p_{\sigma})).
\end{split}
\end{equation}
We then show:

\begin{prop}\label{prop-differential}
Suppose that $A(X)$ is an $h_*$-fibration over
the finite dimensional simplicial complex $X$. 
Then the map 
$d_{\tau}^{\sigma}$ in (\ref{eq-differential}) is zero, if $\Delta^{p-1}_\tau$ is not a 
face of $\Delta^p_{\sigma}$, and we have
$$ d_{\tau}^{\sigma}=(-1)^{{p-l}}\sign(\varphi_{\sigma,\tau}^l)(\res_{\tau}^{\sigma})^{-1},$$
if $\Delta^{p-1}_{\tau}=\Delta^{p-1}_{{\sigma,l}}$,
where we denote by $\res_{\tau}^{\sigma}:h_*(A(\Delta^p_{\sigma}))\stackrel{\cong}{\longrightarrow}
h_*(A(\Delta^{p-1}_{\tau}))$
the isomorphism induced by  the quotient map $A(\Delta^p_{\sigma})\to A(\Delta^p_{\tau})$.
\end{prop}

For the proof of the proposition, we need the following lemma:

\begin{lem}\label{lem-tau-sigma}
Suppose that $\sigma\in C_p$ and $\tau\in C_{p-1}$ are as above. Suppose further
that $Z$ is any closed simplicial sub-complex of $X_p$ such that $\Delta^{p-1}_{\tau}$
and $\Delta^p_{\sigma}$ are contained in $Z$. Then the differential
$d_{\tau}^{\sigma}$ of (\ref{eq-differential}) coincides with the same map, if we
replace $X$ by $Z$.
\end{lem}
\begin{proof} Apply the quotient map $q:A(X)\to A(Z)$ to all ingredients of the composition
of maps in (\ref{eq-tau-sigma}) and use naturality of the long exact sequences in $h_*$-theory.
Since the quotient map induces the identity in the first and in the last place of that chain,
and since, by naturality, all other maps are linked by commutative diagrams (note that
all maps in the chain are maps taken from appropriate long exact sequences linked
by factorizations of the  quotient map $A(X)\to A(Z)$), the result follows.
\end{proof}

\begin{proof}[Proof of Proposition \ref{prop-differential}]
Suppose first that $\Delta^{p-1}_{\tau}$ is not a face of $\Delta^p_{\sigma}$.
To see that $d_{\tau}^{\sigma}$ is then the zero map, we actually show that
the chain of maps of (\ref{eq-tau-sigma}) relative to the sub-complex
$Z=\Delta^{p-1}_{\tau}\cup\Delta^p_{\sigma}$ is the zero map.
We then have $A(Z_{p-1})=A(\Delta^{p-1}_{\tau})\oplus A(\Delta^p_{\sigma}\setminus\iDelta^p_{\sigma})$, and the chain of maps in (\ref{eq-tau-sigma}) becomes the composition
$$h_{{{q+1}}}(A(\iDelta^{p-1}_\tau))\stackrel{\iota_*}{\longrightarrow}
h_{{{q+1}}}(A(\Delta^{p-1}_{\tau}))\oplus h_{{{q+1}}}(A(\Delta^p_{\sigma}\setminus\iDelta^p_{\sigma}))
\stackrel{\partial}{\longrightarrow} h_q(A(\iDelta^p_{\sigma})).$$
But the first map takes its image in the first summand of the middle term, which lies 
in the kernel of the second map.

So we can now restrict to the case where $\Delta^p_{\tau}$ coincides with the 
$l$-th face of $\Delta^p_{\sigma}$. By Lemma \ref{lem-tau-sigma}, we may also assume
without loss of generality that $X=\Delta^p_{\sigma}$.
Using once again naturality of long exact sequences in $h_*$-theory, and applying this to 
the inclusion of the ideal $A(\iDelta^{p-1}_{\sigma,l}\cup \iDelta^p_{\sigma})$ into
$A(\Delta^p_{\sigma}\setminus X_{p-2})$ (which is 
the restriction to the complement of  the ${p-2}$-skeleton of $\Delta^p_{\sigma}$)
shows that in this situation  the chain of maps in (\ref{eq-tau-sigma}) reduces to 
the boundary map $\partial_l$ in the long exact sequence
$$\to h_{{{q+1}}}(A(\iDelta^{p-1}_{\sigma,l}\cup \iDelta^p_{\sigma}))
\to h_{{{q+1}}}(A(\iDelta^{p-1}_{\sigma,l}))\stackrel{\partial_l}{\to} h_q(A(\iDelta^p_{\sigma}))\to .$$
Moreover, it follows from  Proposition \ref{prop-affine}, that replacing the orientation 
of $\Delta^{p-1}_{\tau}=\Delta^{p-1}_{\sigma,l}$ given by $\tau$ to that given by $\sigma_l$
results to the factor $\sign(\varphi_{\sigma,\tau}^l)$ in the  canonical oriented isomorphism
$\Phi_{{{p+q}}}^{p-1}:h_{{{p+q}}}(A(\Delta^{p-1}_\tau))\to h_q(A(\iDelta^{p-1}_\tau))$. Taking this into account,
the map $d_{\tau}^{\sigma}$ becomes $\sign(\varphi_{\sigma,\tau}^l)$-times  the 
composition
$$
h_{{{p+q}}}(A(\Delta^{p-1}_{\sigma,l}))
\stackrel{\Phi_{{{p+q}}}^{p-1}}{\longrightarrow}
h_{{{q+1}}}(A(\iDelta^{p-1}_{\sigma,l}))
\stackrel{\partial_l}{\longrightarrow}
h_q(A(\iDelta^p_{\sigma}))
\stackrel{(\Phi_{{{p+q}}}^p)^{-1}}{\longrightarrow}
h_{{{p+q}}}(A(\Delta^p_{\sigma})).
$$
But it follows from Corollary \ref{cor-lface} that this composition coincides with 
$(-1)^l(\res_{\tau}^{\sigma})^{-1}$ (which plays the role of $\res_l$ in that corollary).
Hence we arrived at the equation
$$d_{\tau}^{\sigma}=\sign(\varphi_{\sigma,\tau}^l)(-1)^{{p-l}}(\res_{\tau}^{\sigma})^{-1},$$
which finishes the proof.
\end{proof}

We now recall the definition of the simplicial cohomology 
of a finite dimensional{and locally finite} simplicial  complex $X$ with coefficients  in a group 
bundle $\G=\bigcup\{G_x:x\in X\}$. 
We refer to Section \ref{sec-bundle} for the definition of a group bundle and for
the notion of the group bundle $\mathcal H_*$ associated to an $h_*$-fibration.

If $Y$ is any simply connected subset of $X$,
a section $g\in \Gamma(Y,\G)$ is said to be {\em constant}, if $g$ becomes constant 
in any trivialization of the bundle over $Y$, which is equivalent to saying that 
$$g(y)=\Phi_{x,y}g(x)\quad\text{for all $x,y\in Y$}.$$
We denote by $\Gamma_{\const}(Y,\G)$ the group of constant sections of $Y$.
It is clear that, if $Y$ is simply connected and $Z\subseteq Y$ is any (simply connected)
subset, then the restriction map 
$$\res_Z^Y: \Gamma_{\const}(Y,\G)\to \Gamma_{\const}(Z,\G)$$
is an isomorphism. In particular, $\Gamma_{\const}(Y,\G)$ is isomorphic to $G_x$ 
for every $x\in Y$.

We now define the $p$-cochains for simplicial cohomology on $X$ with coefficient $\G$
as 
\begin{align*}
C^p(X;\G):=
\{C_p\ni \sigma\mapsto f(\sigma)\in \Gamma_{\const}(\Delta^p_{\sigma}; \G) 
\},
\end{align*}
i.e., as the set of all  maps which assign a 
$p$-simplex $\sigma\in C_p$ to a constant 
section on $\Delta^p_{\sigma}$. Moreover, we define $C^p_{fin}(X;\G)$ as the subgroup 
of $C^p(X;\G)$ consisting of all finitely supported functions.
The boundary map
$$\partial: C^{p-1}(X;\G)\to C^p(X;\G)$$
is defined by
$$\partial f(\sigma)=
\sum_{l=0}^p (-1)^l(\res_{\sigma_l}^{\sigma})^{-1}(f(\sigma_l)),
$$
where we define $f(\sigma_l)=\sign(\varphi_{\sigma,\tau}^l)f(\tau)$, where $\tau\in C_{p-1}$
is the unique element with $\Delta^{p-1}_{\tau}=\Delta^{p-1}_{{\sigma,l}}$, and 
$\varphi_{\sigma,\tau}^l$ is defined as in (\ref{eq-permute}). 
It restricts to boundary map on $C^p_{fin}(X;\G)$.
We define $H^p(X;\G)$ as the $p$-th cohomology of the chain complex 
$(C^p(X;\G),\partial)$ and $H^p_{fin}(X;\G)$ as the $p$-th cohomology
of $(C^p_{fin}(X;\G),\partial)$. Of course, both cohomology groups coincide on finite complexes.
The computations in Proposition \ref{prop-differential} now immediately give

\begin{thm}[Leray-Serre spectral sequence for $h_*$-fibrations]\label{thm-LS}
{\ }
\\
Let $h_*$ be a homology theory on a good category of $C^*$-algebras
and suppose that $X$ is a finite dimensional $\sigma$-finite and  locally finite
simplicial complex. Suppose that $A=A(X)$
is an $h_*$-fibration with associated group bundle $\mathcal H_*$. 
If $h_*$ is $\sigma$-additive or if 
$X$ is finite, the $E_2$-term in the spectral sequence of Proposition \ref{prop-exact1}
is given by $$E_2^{p,q}=H^p_{fin}(X;\H_{{p+q}}).$$
If $h_*$ is $\sigma$-multiplicative, 
then 
$$E_2^{p,q}=H^p(X;\H_{{p+q}}).$$
\end{thm}

\begin{remark} 
{\bf (1)} As mentioned earlier, the main example we have in mind for the above 
theorem is the case where $h_*$ is the $K$-theory functor. But the result applies 
also to other interesting functors like the functor $\KK(B,\cdot)$  for a fixed 
C*-algebra $B$. In general, these functors are only finitely additive, so 
we should restrict to finite simplicial complexes in these situations.

{\bf (2)}  We should note that the cohomology groups $H^p(X;\G)$ 
we defined above coincide with the usual singular cohomology with local coefficients
as defined in many standard text books (e.g., see \cite{DK}), while 
the groups $H^p_{fin}(X;\G)$ are known as
the simplicial cohomology with local coefficients with finite supports. 
\end{remark}

Although we don't want to  go through all the details for the proof
of the Leray-Serre spectral sequence for cohomology theories on 
C*-algebras (the steps are similar as for homology theory with all 
arrows reversed)  we want at least give a proper statement of the result.
As mentioned earlier, $K$-homology serves as a main example of
such theory, but other examples are given by the functors $\KK(\cdot, B)$
for a fixed C*-algebra $B$.

Recall that the simplicial homology $H_p(X;\G)$ with coefficient in a group bundle 
$\mathcal G$ is defined as the homology of the chain complex 
$(C_p(X;\G), d)$, where 
$$
C_p(X;\G):=
\{\sum_{\sigma\in C_p} a_{\sigma}\sigma:  a_{\sigma}\in \Gamma_{\const}(\Delta^p_{\sigma}; \G)
\},
$$
with boundary map $d: C_p(X;\G)\to C_{p-1}(X;\G)$ given by
$$ d(a_{\sigma}\sigma)=\sum_{l=0}^p(-1)^l\sign(\varphi_{\sigma,\tau}^l)\res_{\sigma_i}^{\sigma} (a_{\sigma}) \tau_l
$$
where for each $0\leq l\leq p$, $\tau_l$ is the unique element in $C_{p-1}$ with image
$\Delta^{p-1}_{\sigma, l}$ and $\sign(\varphi_{\sigma,\tau}^l)$ is as before.
Again, if we restrict to finite sums, we obtain the theory $H_p^{fin}(X;\G)$ with finite
supports. The Leray-Serre theorem then reads as  follows

\begin{thm}[Leray-Serre spectral sequence for $h^*$-fibrations]\label{thm-LS1}
{\ }
\\
Let $h^*$ be a cohomology theory on a good category of $C^*$-algebras
and let  $X$ be as in Theorem \ref{thm-LS}. Suppose that $A=A(X)$
is an $h^*$-fibration with associated group bundle $\mathcal H^*$. 
If $h^*$ is $\sigma$-additive or if 
$X$ is finite, the $E^2$-term in the spectral sequence of Proposition \ref{prop-exact1}
is given by $$E^2_{p,q}=H_p^{fin}(X;\H_{{p+q}}).$$
If $h^*$ is $\sigma$-multiplicative, 
then 
$$E^2_{p,q}=H_p(X;\H_{{p+q}}).$$
\end{thm}

We want to close this section with a discussion how the spectral sequences 
considered here give new invariants for $\RKK$-equivalence of C*-algebra bundles.
The following lemma might be well known to the experts,
but since we rely heavily on it, we give the argument here.
For notation, we let $\CKK$ denote the category whose objects are 
separable C*-algebras and the morphisms between two objects $A$ and $B$ are
the elements in $\KK(A,B)$. Recall that for any pair of bundles $A(X)$ and $B(X)$
and any continuous{inclusion} map $f: Y{\hookrightarrow}  X$ there exists a canonical pull-back map
$$f^*: \RKK(X; A(X), B(X))\to \RKK(Y; f^*A(Y), f^*B(Y)).$$
In particular, if $Y\subseteq X$, then we obtain restrictions $\frak x\mapsto \res_Y^X(\frak x)$
from $\RKK(X, A(X), B(X))\to \RKK(Y; A(Y), B(Y))$ given via the inclusion of $Y$ into $X$
(see \cite[Proposition 2.2]{Kas2}). Recall also that 
a short exact extension $0\to J\to A\stackrel{q}{\to} A/J\to 0$ of C*-algebras is \emph{semi-split}
if there exists a completely positive  section $s: A/J\to A$. Note that this is always true
if $A/J$ is nuclear (which follows if $A$ is nuclear).

\begin{lem}\label{lem-RKK}
Suppose that $A(X)$ and $B(X)$ are C*-algebra bundles over $X$ and 
suppose that $U\subset X$ is open. Let $\frak x\in \RKK(X;A,B)$
and let  $i_A: A(U)\to A(X)$ and $q_A:A(X)\to A(X\setminus U)$ denote the inclusion 
and quotient maps (and similarly for $B(X)$). Then the diagram
$$\begin{CD}
A(U) @>i_A>> A(X) @>q_A>>  A(X\setminus U)\\
@V\res_U^X(\frak x) VV   @V\frak x VV   @VV\res_{X\setminus U}^X(\frak x)V\\
B(U) @>i_B>> B(X) @>q_B>>  B(X\setminus U)
\end{CD}
$$
commutes in the category $\CKK$. Moreover, if both extensions in the above diagram
are semi-split (we do not require that the c.p. sections are $C_0(X)$-linear), then the diagram
$$\begin{CD}
SA(X\setminus U) @>\partial_A >>   A(U)\\
@V\res_{X\setminus U}^X(\frak x) VV   @VV\res_U^X(\frak x) V   \\
SB(X\setminus U) @>\partial_B >> B(U)
\end{CD}
$$
also commutes in $\CKK$, where $SA$ and $SB$ denote the suspensions of $A$ and $B$,
respectively.
\end{lem}
\begin{proof} Let $(E(X),T)$ be a Kasparov cycle representing $\frak x$.
We may assume that $A(X)$ acts nondegenerately on $E(X)$.
If $f:Y\to X$ is a continuous map, then $f^*(\frak x)$ is represented by the cycle
$$(C_0(Y)\otimes_{C_0(X)}E(X), 1\otimes T)\cong (E(X)\otimes_{C_0(X)} C_0(Y), T\otimes 1)$$
depending on whether we want tensor $C_0(Y)$ from the left or from the right.
The first square of the first diagram then follows from an obvious isomorphism of 
$\KK(A(U), B(X))$-cycles
$$(A(U)\otimes_{A(X)}E(X), 1\otimes T)\cong (C_0(U)\otimes_{C_0(X)} E(X), 1\otimes T).$$
The left cycle represents $[i_A]\otimes_{A(X)}\frak x$
and the right cycle represents $\res_U^X(\frak x)\otimes_{B(U)}[i_B]$.

Similarly, the second square of the first diagram follows from the observation that both products 
$[q_A]\otimes_{A(X\setminus U)}\res_{X\setminus U}^X(\frak x)$ and 
$\frak x\otimes_{B(X)}[q_B]$ are represented by the module
$(E(X)\otimes_{C_0(X)}C_0(X\setminus U), T\otimes 1)$ with the canonical module actions.

So let us now assume that both extensions are semi-split. Then the 
boundary map $\partial_A$  in the second diagram is given 
by Kasparov product with an element
$[\partial_A]\in \KK_1(A(X\setminus U), A(U))$ 
constructed as follows:
let 
$$Z=\big((0,1]\times (X\setminus U)\big) \cup\big( \{1\}\times U\big)\subseteq (0,1]\times X.$$
Let $p:Z\to X$ denote the canonical projection and write $A(Z)$ for the pull-back $p^*A(Z)$.
Note that, as an algebra, $A(Z)$ is just the mapping cone $C_{q_A}$
{of the homomoprhism  $q_A$}.
Let  $e_A: A(U)\to A(Z)$ be the inclusion map given by identifying $U$ with the open set
$\{1\}\times U\subseteq Z$. It is shown in \cite[Theorem 19.5.5]{Black} that 
$e_A$ is a $\KK$-equivalence.
Let $u_A$ denote its inverse and let 
$j_A: SA(X\setminus U)=A\big((0,1)\times (X\setminus U)\big)\to A(Z)$ denote the inclusion.
Then it is shown in \cite[Theorem 19.5.7]{Black} that $[\partial_A] =[j_A]\otimes_{A(Z)} u_A$.
The same construction applies to $B(X)$. Let $p^*(\frak x)$ be the pull-back of $\frak x$
in $\RKK(Z; A(Z), B(Z))$. Then the commutativity of the second diagram follows from the
commutativity of 
$$
\begin{CD}
A\big((0,1)\times (X\setminus U)\big) @>j_A>> A(Z)  @<e_A<< A(U)\\
@V\res_{(0,1)\times (X\setminus U)}^Z(p^*(\frak x)) VV  
@V p^*(\frak x)VV  @VV \res_{U}^Z(p^*(\frak x))V\\
B\big((0,1)\times (X\setminus U)\big) @>j_B>> B(Z)  @<e_B<< B(U)
\end{CD}
$$
which is a consequence of the commutativity of the first diagram in the lemma.
\end{proof}

We say that a (co-)homology theory on $\Csep$ is \emph{$\KK$-representable}
if there exists a C*-algebra $B$ such that the (co-)homology theory is given 
by $A\mapsto \KK_*(B,A)$ (resp. $A\mapsto \KK(A,B)$). Of course $K$-theory 
and $K$-homology are important examples, but also $K$-theory with coefficients 
in $\ZZ/n\ZZ$ is an example of such theory. Note that every $\KK$-fibration is automatically
an $h_*$-fibration (resp. $h^*$-fibration) if $h_*$ (resp $h^*$) is $\KK$-representable.

\begin{cor}\label{cor-spectral}
Let $h_*$ be a $\KK$-representable homology theory on $\Csep$.
Assume that $A(X)$ and $B(X)$ are $h_*$-fibrations.
Then any class $\frak x\in \RKK(X;A,B)$ induces a morphism 
between the associated group bundles $\mathcal H_*(A)$ and $\mathcal H_*(B)$.
If $X$ is a CW-complex and $A(X)$ and $B(X)$ are nuclear, 
then $X$ induces a morphism between the associated exact couples
for the Leray-Serre spectral sequence. 

In particular, if $\frak x$ is a $\RKK$-equivalence, then it induces an isomorphism
between the Leray-Serre spectral sequences for $h_*(A)$ and $h_*(B)$.
A similar statement holds for $\KK$-representable cohomology theories. 
\end{cor}

\begin{remark}\label{rem-RKK}
Note that in the case of $K$-theory one can omit the nuclearity assumption on $A(X)$ and $B(X)$
in the above lemma. The reason is that for $U\subseteq X$ open, and $Z$
as in the proof of Lemma \ref{lem-RKK}, we always have an isomorphism 
$ K_*(A(U))\stackrel{e_*}{\cong} K_*(A(Z))$
such that the boundary map $\partial_A: K_{*+1}(A(X\setminus U))\to K_*(A(U))$
is given via the composition 
$\partial_A=(e_*)^{-1}\circ j_{A,*}$, with $e_A$ and $j_A$ as in the proof of the lemma.
Thus the same argument as in the lemma shows that for each such $U$ the transformation 
given by Kasparov product with the appropriate restrictions of $\frak x$ gives a transformation
between the $K$-theory long exact sequences for $A$ and $B$ corresponding to $U$.
This is all we need to obtain a well-defined  morphism between the exact couples.
\end{remark}

%



\section{Applications to non-commutative torus bundles}

Recall from \cite{ENO} that  a non-commutative {principal} 
$\TT^n$-bundle (or NCP $\TT^n$-bundle for short) is defined as 
a C*-algebra bundle $A(X)$ equipped with a fibrewise action of 
$\TT^n$ such that $A(X)\rtimes \TT^n$ is isomorphic to $C_0(X,\K)$.
By Takesaki-Takai duality, every such bundle can be realized up to stabilization 
by a crossed product $C_0(X,\K)\rtimes \ZZ^n$ for some fiberwise action of 
$\ZZ^n$ on $C_0(X,\K)$. Using results from \cite{EW1,EW2} we showed
in \cite[\S 2]{ENO} that the $\TT^n$-equivariant stable isomorphism classes 
of  NCP $\TT^n$-bundles over a given space $X$ can be classified by 
the pairs $([Y], f)$, where $[Y]$ denotes the isomorphism class of a classical
{principal} $\TT^n$-bundle $Y\stackrel{q}{\to} X$ 
and $f: X\to \TT^{\frac{n(n-1)}{2}}$ is a continuous map. 
The NCP $\TT^n$-bundle corresponding to the pair $([Y], f)$ is then given by
$$Y*\big(f^*(C^*(H_n))\big)=\big(C_0(Y)\otimes_{C_0(X)} f^*(C^*(H_n))\big)^{\TT^n}.$$
Let's recall the ingredients of this construction:
$C^*(H_n)$ is the group C*-algebra of the group
$$H_n=\lk g_1,\dots,g_n,  f_{ij}: 1\leq i<j\leq n\rk$$
with relations $g_ig_j=f_{ij}g_jg_i$ and $f_{ij}$ central for all $1\leq i<j\leq n$.
This group has center $Z_n=\lk f_{ij}: 1\leq i<j\leq n\rk \cong \ZZ^{\frac{n(n-1)}{2}}$
and, therefore, $C^*(H_n)$ is a continuous C*-algebra bundle over 
$\TT^{\frac{n(n-1)}{2}}\cong \widehat{Z_n}$ via the inclusion 
$$\Phi: C(\TT^{\frac{n(n-1)}{2}})\cong C^*(Z_n)\to Z(C^*(H_n)).$$
Moreover, if we  equip $C^*(H_n)$ with the dual action of $\TT^n\cong \widehat{H_n/Z_n}$ it
becomes an NCP $\TT^n$-bundle over
$\TT^{\frac{n(n-1)}{2}}$. We shall denote by $U_1,\dots,U_n$ the
unitaries of $C^*(H_n)$ corresponding to $g_1,\dots,g_n$, respectively, and by
$W_{i,j}$ the unitaries corresponding to $f_{i,j}$ for $1\leq i<j\leq n$.
If $f:X\to \TT^{\frac{n(n-1)}{2}}$ is a continuous map, then  the pull-back $f^*(C^*(H_n))$
becomes an NCP $\TT^n$-bundle over $X$. By taking the $C_0(X)$-balanced tensor product
of $f^*(C^*(H_n))$ with $C_0(Y)$ (with $C_0(X)$-action on $C_0(Y)$ induced by $q:Y\to X$
in the obvious way), and then taking the algebra of fixed points with respect to the 
diagonal action {(with action by the inverse automorphism on one factor)} provides the NCP $\TT^n$-bundle $Y*\big(f^*(C^*(H_n))\big)$.
The $\TT^n$-action is induced by the given $\TT^n$-action on $Y$.

In \cite{ENO} we studied the topological nature of the C*-algebra bundles $A(X)$ 
after ``forgetting'' the underlying $\TT^n$-actions. In particular, we were interested 
in the question under what conditions two such bundles are $K$-theoretically equivalent
fibrations, i.e., when are two such bundles $\RKK$-equivalent.
We arrived at the following result:

\begin{thm}[{\cite[Theorem 7.2]{ENO}}]\label{thm-ENO}
Let $A(X)$ be a NCP $\TT^n$-bundle over the path connected space $X$ and let $f:X\to \TT^{\frac{n(n-1)}{2}}$ be the continuous map associated to $A(X)$ as above. Then the 
following are equivalent:
\begin{enumerate}
\item $f$ is homotopic to a constant map.
\item The $K$-theory bundle $\mathcal K_*(A(X))$ is trivial.
\item $A(X)$ is $\RKK$-equivalent to $C_0(Y)$ for some (commutative)  {principal}
$\TT^n$-bundle $q:Y\to X$.
\end{enumerate}
\end{thm}

In this section we will use the Leray-Serre spectral sequence to obtain the following triviality result.

\begin{thm}\label{thm-triv}
Let $A(X)$ be the NCP $\TT^n$-bundle corresponding to the pair $([Y], f)$ as 
explained above such that $X$ is a finite dimensional {locally finite and $\sigma$-finite}
simplicial complex. Then the following are equivalent:
\begin{enumerate}
\item $A(X)$ is $\RKK$-equivalent to $C_0(X\times \TT^n)$.
\item The $K$-theory bundle $\mathcal K_*(A)$ is trivial and all $d_2$-differentials
in the Leray-Serre spectral theorem for $A(X)$ vanish.
\item $f$ is homotopic to a constant map and $Y\cong X\times \TT^n$ as $\TT^n$-bundles.
\end{enumerate}
\end{thm}

The proof depends on explicit calculations of the $d_2$ maps in the spectral 
sequence of a commutative principal $\TT^n$-bundle $q:Y\to X$, and then 
transporting this result to the spectral sequence for the  {$K_*$-fibration} $A(X)$.
In what follows we shall always denote by $Y_x$ the fibre of a given 
{principal} $\TT^n$-bundle over a point $x\in X$ and we write
$\{E^{p,q}_r(Y),d_r\}$  for the spectral sequence corresponding to a fixed triangulation 
of the base $X$. We always assume that $X$ is finite dimensional and $\sigma$-finite.
The following proposition is certainly well-known to the experts, but since we 
didn't find an appropriate reference we give a proof.

\begin{prop}\label{prop-trivial}
Let $Y\stackrel{q}{\to}X$ be a principal $\TT$-bundle. Then the
$(0,1)$-degree component of the differential 
$$d_2^{0,1}:H^0(X,K^{1}(Y_x))\to H^2(X,K^{0}(Y_x))$$ on $E^{*,*}_2(Y)$
vanishes if and
only if $Y\stackrel{q}{\to}X$ is trivial.
\end{prop}
The proof of this proposition will require some preliminary work. 
Let $X_{-1}=\emptyset\subset X_0\subset\cdots\subset X_n=X$ be the
skeleton decomposition of $X$ and let us set $Y_k=q^{-1}(X_k)$. For $\sigma$
a simplex of $X$, let us denote by $V_\sigma$ the closure of the
$*$-neighboorhood of $\sigma$. {We may assume that}
\begin{itemize}
\item for all simplices $\sigma$ there exist
continuous maps $\Psi_\sigma:q^{-1}(V_\sigma)\to\TT$;
\item for all pairs of simplices $\sigma$ and $\sigma'$  which are faces of a common simplex,
 there exists a continuous map $h_{\sigma,\sigma'}:V_\sigma\cap
 {V_{\sigma'}}\to \RR$
\end{itemize}
such that
\begin{enumerate}
\item $q^{-1}(V_\sigma)\to V_\sigma\times \TT;y\mapsto
 (q(y),\Psi_\sigma(y))$ is a $\TT$-equivariant homeomorphism.
\item $\Psi_{\sigma'}=\Psi_\sigma\cdot e^{2i\pi h_{\sigma,\sigma'}}$ on $V_\sigma\cap
{ V_{\sigma'}}$  for all pairs of simplices $\sigma$ and $\sigma'$  which
 are faces of a common simplex.
\end{enumerate}
We will denote by $\V$ the corresponding atlas. Notice that this atlas
provides an idenfication $K^1(Y_x)\cong K^1(\TT)$ induced by
$Y_x\to\TT;\,y\mapsto \Psi_\sigma(y)$ for any {$x$ in the} simplex $\sigma$.
{Although the identification $Y_x\cong\TT$ depends on $\sigma$, it follows from (ii) that
the induced map
$K^1(Y_x)\cong K^1(\TT)$ does not.}


Let {$X_0=\{x_0,x_1, x_2, x_3, \ldots\}$ be} the set of vertices of $X$. If $x_i$
and $x_j$ are connected by an edge, we will denote by $e_{i,j}$ the
oriented edge starting at $x_i$ and ending at $x_j$. For 
$t\in [0,1]$, we define  $x_{i,j}(t)=(1-t)x_i+tx_j$ in $e_{i,j}$. Let
$U_\V:Y_0\to\TT$ be the continuous map such that $U_\V$ and
$\Psi_{x_i}$ coincide on $q^{-1}(x_i)$. We extend $U_\V$ to a
continuous map $W_\V:Y_1\to\TT$ in the following way: If $x_i$
and $x_j$ are connected by the oriented edge $e_{i,j}$, we define 
$W_\V$ on $q^{-1}(e_{i,j})$ by 
\begin{equation*}
W_\V(z)=e^{2i\pi t h_{x_i,x_j}(q(z))}\Psi_{x_i}(z)
\end{equation*} for $q(z)=x_{i,j}(t)$.
We have 
\begin{eqnarray*}
E_1^{0,1}(Y)&=&K^1(Y_0\setminus Y_{-1})=K^1(Y_0)\\
E_1^{1,2}(Y)&=&K^0(Y_1\setminus Y_{0})
\end{eqnarray*} and 
$d_1^{0,1}:K^1(Y_0)\to K^0(Y_1\setminus Y_{0})$ is the boundary map
$\partial$ 
associated the the pair $(Y_0,Y_1)$. Since $U_\V$ is the restriction
of $W_\V$ to $Y_0$, the class $[U_\V]$ of $U_\V$ in $K^1(Y_0)$
satisfies $\partial[U_\V]=0$ and thus $[U_\V]$ defines a class
$\omega_\V$ in
$E_2^{0,1}(Y)\cong H^0(X,K^1(Y_x))\cong \ZZ$ which is thereby a
generator.
\begin{lem}\label{lem-chern}
With the notations  above and up to the canonical identification
$K^{0}(Y_x)\cong\ZZ$ (which sends $[1]$ to $1$)
$d_2^{1,0}\omega_\V\in H^2(X,K^{0}(Y_x))\cong H^2(X,\ZZ)$ is
the first Chern class of $Y$.
\end{lem}
\begin{proof}
We   extend $W_\V:Y^1\to\TT$ to a continuous map $\phi_\V:Y_2\to\CC$
given on $Y_{\sigma}=q^{-1}(\sigma)\cong \sigma\times \TT$ for
a $2$-simplex $\sigma$ in $X$ with boundary $\partial\sigma$ and center
$x_\sigma$ by $\phi_\V(tx+(1-t)x_{\sigma},z)={tW_{\V}(x,z)}$ for $t$ in $[0,1]$,
$x$ in $\partial\sigma$ and $z\in \TT$.
Then 
$$V_{\V}:=\begin{pmatrix}
\phi_\V& -(1-\phi_\V\bar{\phi}_\V)^{1/2}\\
(1-\phi_\V\bar{\phi}_\V)^{1/2}&\bar{\phi}_\V
\end{pmatrix}$$
is a lift in $U_2(C(Y_2))$ for  $\left(\begin{smallmatrix}
W_\V& 0\\
0&\overline{W}_\V
\end{smallmatrix}\right)$
and thus 
$\partial[W_\V]=[P_\V]-\left[\left(\begin{smallmatrix}
1& 0\\
0&0
\end{smallmatrix}\right)\right]$, where
$$
P_\V=
V_{\V}\cdot
\begin{pmatrix}
1& 0\\
0&0
\end{pmatrix}
\cdot V_\V^*
=\begin{pmatrix}
{\phi}_\V\bar{\phi}_\V&\phi_\V(1-\phi_\V\bar{\phi}_\V)^{1/2}\\
\bar{\phi}_\V(1-\phi_\V\bar{\phi}_\V)^{1/2}&1-{\phi}_\V\bar{\phi}_\V
\end{pmatrix}
$$
is a projector in $M_2(C(Y_2))$.
Then, {up to Bott periodicity}, $d_2^{1,0}\omega$ is the class in  $H^2(X,K^{0}(Y_x))=E^{2,0}_2(Y)$
of the simplicial $2$-cocycle $c$, with value on a $2$-simplex
$\sigma$ {oriented by its boundary $\partial\sigma$}
$$c(\sigma)=i_\sigma^*\circ\partial[W_\V]\in K^0(q^{-1}(\intsi))\cong
K^0(\intsi\times\TT),$$ where $\intsi=\,\sigma\setminus\partial\sigma$ and
$i_\sigma$ is the inclusion  map 
$i_\sigma:\intsi\hookrightarrow Y_2\setminus Y_1$. Finally, we get 
$$c(\sigma)=[P_\V|\sigma]-\left[\left(\begin{smallmatrix}
1& 0\\
0&0
\end{smallmatrix}\right)\right]\in K^0(q^{-1}(\intsi)),$$
where 
$P_\V|\sigma\in M_2(C(q^{-1}(\sigma))$ is the  restriction of  $P_\V$ to 
$q^{-1}(\sigma)$. Let us denote by $\Phi_\sigma$ the inverse of the
trivialisation map $q^{-1}(\sigma)\to\sigma\times\TT;\,y\mapsto
(q(y),\Psi_\sigma(y))$. We then  get isomorphisms
\begin{equation}\label{equ-eval}
K^0(q^{-1}(\intsi))\cong K^0(\intsi)\cong \ZZ,
\end{equation}
where 
\begin{itemize}
\item the first isomorphism is induced by $\intsi\to q^{-1}(\intsi);\,
x\mapsto \Phi_\sigma(x,1)$ {(if we identify $q^{-1}(\intsi)$
with   $\intsi\times\TT$  via
$\Psi_\sigma$, this simply becomes 
$\intsi\to\intsi\times\TT;\,x\mapsto (x,1)$)};
\item the second map is the Bott periodicity for the interior
 $\intsi\cong\RR^2$ of the
 oriented simplex $\sigma$.
\end{itemize}
Let us define for a continuous map $f:\sigma\to\CC$ with $|f|\leq 1$ 
and $|f|=1$ on $\partial\sigma$ the projector
$$P_f=\begin{pmatrix}
|f|^2&f(1-|f|^2)^{1/2}\\
\bar{f}(1-|f|^2)^{1/2}&1-|f|^2
\end{pmatrix}.$$
{Then $[P_f]-\left[\left(\begin{smallmatrix}
1& 0\\
0&0
\end{smallmatrix}\right)\right]$ is the image of
$[f|_{\partial\sigma}]\in K^1(\partial\sigma)$ under the boundary map in
$K$-theory associated to the exact sequence 
$$0\to C_0(\intsi)\to C(\sigma)\to C(\partial\sigma)\to 0.$$ In particular  $[P_f]-\left[\left(\begin{smallmatrix}
1& 0\\
0&0
\end{smallmatrix}\right)\right]$ only depends on the winding number of
$f|_{\partial\sigma}$ on $\partial\sigma$ and this winding number is
precisely the image of  $[P_f]-\left[\left(\begin{smallmatrix}
1& 0\\
0&0
\end{smallmatrix}\right)\right]$  under the second isomorphism of equation
\ref{equ-eval}.} 
Consequently, if we set
$$f_{\V,\sigma}:\sigma\to\CC;\, x\mapsto  \Phi_\sigma(x,1),$$ then the
image of $c(\sigma)$ under the chain of isomorphism of equation
\ref{equ-eval}
is the winding number of the restriction of $f_{\V,\sigma}$ to the
oriented boundary $\partial\sigma$.

If $\sigma$ has vertices
$x_i,x_j$ and $x_k$ connected by oriented edges $e_{i,j},e_{j,k}$ and
{$e_{k,i}$}, then since $h_{\sigma,x_i}+h_{x_i,x_j}-h_{\sigma,x_j}$ and
 $h_{\sigma,x_j}+h_{x_j,x_k}-h_{\sigma,x_k}$ are integers, we have
{$f_{\V,\sigma}=e^{2i\pi h_{\sigma,x_i}}\cdot g_{\V,\sigma}$} where
\begin{eqnarray*}
g_{\V,\sigma}(x_{i,j}(t))&=&\exp (2i\pi th_{x_i,x_j}(x_{i,j}(t)))\\
g_{\V,\sigma}(x_{j,k}(t))&=&\exp (2i\pi(h_{x_i,x_j}(x_{j,k}(t))+
th_{x_j,x_k}(x_{j,k}(t))))\\
g_{\V,\sigma}(x_{k,i}(t))&=&\exp (2i\pi(h_{x_i,x_j}(x_{k,i}(t))+
h_{x_j,x_k}(x_{k,i}(t))+th_{x_k,x_i}(x_{k,i}(t)))).
\end{eqnarray*}
But {$e^{2i\pi h_{\sigma,x_i}}$} has winding number $0$, and
$h_{x_i,x_j}$ and $h_{x_j,x_k}$ can be pushed forward homotopically to
the edge {$e_{k,i}$}, and thus  the restriction of $g_{\V,\sigma}$ to $\partial\sigma$ is
a unitary  homotopic to 
\begin{eqnarray*}
x_{i,j}(t)&\mapsto&1\\
x_{j,k}(t)&\mapsto&1\\
x_{k,i}(t)&\mapsto&\exp (2i\pi t(h_{x_i,x_j}(x_{k,i}(t))+
h_{x_j,x_k}(x_{k,i}(t))+h_{x_k,x_i}(x_{k,i}(t)))).
\end{eqnarray*}
Since $h_{x_i,x_j}+
h_{x_j,x_k}+h_{x_k,x_i}$ is an integer $m_{\sigma}$,  the
restriction to $\partial\sigma$ of $g_{\V,\sigma}$ and hence of $f_{\V,\sigma}$ has
winding number $m_{\sigma}$. Up to the composition of the two
isomorphisms of equation \ref{equ-eval}, we finally get that
$c(\sigma)=m_{\sigma}$, which is precisely the cocycle defining the first
Chern class of $Y$.
\end{proof}

\begin{proof}[Proof of Proposition \ref{prop-trivial}] Since $\omega_\V$ is
a generator for $E^{0,1}_2(Y)\cong\ZZ$, we see from  Lemma \ref{lem-chern}
that $d_2^{0,1}:E^{0,1}_2(Y)\to E^{2,0}_2(Y)$ is vanishing if and only if the
first Chern class of $Y$ vanishes, i.e if and only if $Y$ is trivial.
\end{proof}

Let us now generalise this result to $\TT^n$-principal bundles. For this, we
define  $$\TT^n_i=\{(z_1,\ldots,z_n)\in\TT\text{ such that }z_i=1\}$$
and we let $Y_i=Y/\TT^n_i$ denote the quotient space for the action of 
$\TT^n_i$ on {a principal  $\TT^n$-bundle  $Y\to X$}. Then $Y_i$ is a $\TT$-principal bundle with
action induced by the inclusion $\alpha_i:\TT\hookrightarrow\TT^n$ of
the $i$-th factor and 
with base space $X$. Moreover $Y$ is isomorphic as a $\TT^n$-bundle to
$\alpha_1^*Y_1*\cdots \alpha_n^*Y_n$, where $\alpha_i^*Y_i$ is the
$\TT^n$-bundle induced from $Y_i$ by $\alpha_i$. In consequence, $Y$
is completly determined (up to isomorphism of $\TT^n$-bundle) by the
first Chern classes of the {principal $\TT$-bundle} $Y_i$ and $Y$
 is {a trivial $\TT^n$-bundle} if and only if all
the $Y_i$  are trivial.
\begin{prop}\label{prop-Tn-trivial} 
Let $Y\stackrel{q}{\to}X$ be a principal $\TT^n$-bundle. Then the
$(0,1)$-degre component of the differential
$$d_2^{0,1}:H^0(X,K^{1}(Y_x))\to H^2(X,K^{0}(Y_x))$$ on $E^{*,*}_2(Y)$
vanishes if and
only if $Y\stackrel{q}{\to}X$ is trivial.
\end{prop}
\begin{proof}
Since $K^*(Y)$ is equipped with an algebra structure, $d_2$ is a map
of differential algebra. The  unital algebra $K^*(Y_x)$ being generated by
the image of $K^*(Y_{i,x})$ under the morphism induced by the
projection map $Y\to Y_i$, the map $d_2^{0,1}$ is completly determined
by the image of elements coming from $K^*(Y_{i,x})$.
The projection map $Y\to Y_i$ provides a morphism of spectral
sequences
$(E^{p,q}_r(Y_i),d_r) \longrightarrow
(E^{p,q}_r(Y),d_r)$.
In particular, we get a
commutative diagram
$$\begin{CD} H^0(X,K^{1}(Y_{x,i}))  @>d_2^{0,1}>>  H^2(X,K^{0}(Y_{x,i}))\\
@VVV      @VVV\\
H^0(X,K^{1}(Y_x))  @>d_2^{0,1}>>  H^2(X,K^{0}(Y_x)),
\end{CD}$$where the vertical arrows are induced by the projection  of {the} fiber
{$\pi_i:Y_{x}\rightarrow Y_{i,x}$}. 
Since the inclusion $K^0(Y_{i,x})\to K^0(Y_x)$ is injective (since it sends $[1]$ to $[1]$)
the right vertical map is injective, too. Thus
it follows from  Proposition \ref{prop-trivial} that the
range of the left vertical map of the diagram lies
in the kernel of $d_2^{0,1}:H^0(X,K^{1}(Y_x))\to H^2(X,K^{0}(Y_x))$ if
and only if $Y_i$ is trivial.
\end{proof}

\begin{remark}\label{rem-triv}\
\begin{enumerate}
\item If $Y\cong X\times \TT^n$ is the trivial $\TT^n$-bundle over $X$, then
$d_2:E_2^{*,*}\to E_2^{*,*}$ vanishes completely. To see this, we
can use the K\"unneth formula in $K$-theory to show
that it is enough to 
prove the result for the Leray-Serre spectral sequence associed to $C_0(X)$,
i.e., the Hirzebruch spectral sequence for the $K$-theory of $X$. Then
$E^{p,q}_2(X)=H^p(X,\ZZ)$ if $p-q$ is even and $E^{p,q}_2(X)=0$ if
$p-q$ is odd. Since $d_2$ maps $E^{p,q}_2(X)$ to $E^{p+2, q+1}_2(X)$,
it follows that either the source or the target of this map must be zero.
Thus $d_2=0$. As a direct consequence of this observation and of Proposition \ref{prop-Tn-trivial}
we now see that a principal $\TT^n$-bundle $Y\stackrel{q}{\to}X$ is trivial if
and only if all $d_2$-differentials in the associated spectral
sequence vanish.
\item {More generally}, if $Y\stackrel{{q}}{\to} X$ is a {principal} $\TT^n$-bundle, then
   $K^*(Y)$ is endowed with a $K^*(X)$-module  structure. Then
   $\oplus E^{p,q}_2(Y)=\oplus H^p(X,K^{p-q}(Y_x){)}$ has a $\oplus
   H^p(X,K^{p-q}(\{*\}){)}$-module  structure provided by the cup
   product. Thereby,  $\oplus H^0(X,K^{p-q}(Y_x){)}$ being a generating
   set and since the differential  $d_2$ is $\oplus
   H^p(X,K^{p-q}(\{*\}){)}$-linear, then $d_2$ is completly determined
   by the Chern classes $c_i$.
\end{enumerate}
\end{remark}

\begin{proof}[Proof of Theorem \ref{thm-triv}]
Using  the fact that $\RKK$-equivalence induces 
an equivalence of spectral sequences, the result is 
now a direct consequence of Theorem \ref{thm-ENO} and the above remark.
\end{proof}

A natural question is then: 
let $A(X)$ be  a NCP-$\TT^n$-bundle  with classifying data
{$([Y], f)$}. Can we recover from the { exact sequence
$\{E_r^{*,*}(A)),d_r\}$ derived from the {$K_*$-fibration} $A(X)$} any information concerning $Y$?
{As we shall see below  this is not  always the case.}

{ Using the above notations, we let $q_i:Y\to
Y_i=Y/\TT^n_i$ denote the quotient map. There are canonical $C_0(X)$-linear 
$*$-homomorphisms $\Lambda_i:C(Y_i)\to A(X)=Y*f^*(C^*(H))(X)$ given as follows:
First we define a 
$*$-homomorphism $\tilde{\Lambda}_i:C_0(Y_i)\to \big(C_0(Y)\otimes C({\TT}))^{\TT^n}$ by 
$\big(\tilde{\Lambda}_i(\phi)\big)(y,t)=\phi(tq_i(y))$, where
{$\TT$ is acted upon by $\TT^n$ using the  projection on the
$i$-th component.} 
Next we identify $C(\TT_i)$ with $C^*(U_i)$ via functional calculus to
obtain from this a  
well defined $C_0(X)$-linear 
$*$-homomorphism (also called $\tilde{\Lambda}_i$) from 
$C_0(Y_i)$ to  $\big(C_0(Y)\otimes C^*(U_i)\big)^{\TT^n}$.
Since $U_i$ commutes with all $W_{kj}$, the $C^*$-algebra $C^*\langle
U_i,\,W_{kj};\,1\leq k<j\leq n\rangle$ is a $C(\TT^{n(n-1)/2})$-subalgebra of
$C^*(H_n)$ isomorphic to $C(\TT^{n(n-1)/2})\otimes  C^*(U_i)$ and
hence, as $C_0(X)$-algebras,
we can identify  
$\left(C_0(Y)\otimes
C^*(U_i)\right)^{\TT^n}\cong \left(C_0(Y)\otimes_{C_0(X)}C(\TT^{n(n-1)/2})\otimes  C^*(U_i)\right)
)^{\TT^n}$ with a subalgebra of
$
Y*f^*(C^*(H_n))=\left(
 C_0(Y)\otimes_{C_0(X)}f^*{C^*(H_n)}\right)
^{\TT^n}
$.  
The map $\Lambda_i$ is then given by the composition of
$\tilde{\Lambda}_i$ with this 
inclusion.}

{The morphism $\Lambda_i:C_0(Y_i)\to A(X)$} induces a morphism
$$\{\Lambda_{i,r}^{p,q}:{E_{r}^{p,q}(Y_i)\to E_{r}^{p,q}(A)}\}$$ of spectral
sequences. At the $E_2$-term,  the   morphism
$$\Lambda_{i,2}^{p,q}:H^p(X,K^{p-q}(\TT))\to
H^p(X,\mathcal{K}_{p-q}(A))$$ is induced by
the morphism of group bundles $$(\Lambda_{i,x,*})_{x\in X}: X\times
K^*(\TT)\to \mathcal{K}_*(A).$$ In particular, if $c_i\in H^2(X,\ZZ)$
is the Chern class of $Y_i$ then using the notations of Lemma \ref{lem-chern} {we get}
\begin{equation}\label{equ-d2cc}
d_2^{0,1}(\Lambda_{i,2}^{0,1}([\omega_\nu])=\Lambda_{i,2}^{2,0}(c_i).
\end{equation}

According to \cite{DK}, if $X$ is path connected  with base point $x$,
then the
cohomology group $H^*(X,\mathcal{K}_{*}(A))$ can be described in the
following way: fix  a simplicial decomposition of  $X$
and lift it to a $\pi_1(X)$-invariant  simplicial decomposition of
$\widetilde{X}$. Let $\SS_*(\widetilde{X})$ be the simplicial complex obtained from this
simplicial decomposition of $\widetilde{X}$. Then $\SS_*(\widetilde{X})$ is endowed with an
action of $\Gamma=\pi_1(X)$ {by automorphisms} and $H^*(X,\mathcal{K}_{*}(A))$ is then
the  cohomology
of the complex $\operatorname{Hom}_{\Gamma}(\SS_*(\widetilde{X}),K_*(A_x))$ of 
$\Gamma$-equivariant homomorphisms from $\SS_*(\widetilde{X})$ to
$K_*(A_x)$.
In particular, in  degree zero {we get}
$$H^0(X,\mathcal{K}_{*}(A))=\operatorname{Inv}_{\Gamma}K_*(A_x),$$
where for  an abelian group  $N$ equipped with an action of $\Gamma$,
$\operatorname{Inv}_{\Gamma}\,N$ stands for the set of
$\Gamma$-invariant elements of
$N$. Since we will need it later on, we can also define at this point the
coinvariant for  $N$ to be
$\operatorname{Coinv}_{\Gamma}\,N=N/{\langle x-\gamma x;x\in N\text{ and
}\gamma\in\Gamma\rangle}$.

Using this, and  noticing   that the classes $\{[U_{1,x}],\ldots,[U_{n,x}]\}$
of  $K_1(A_x)$ are invariant, we get that
$\Lambda_{i,2}^{0,1}([\omega_\nu])=[U_{i,x}]$ and thus according to
equation \ref{equ-d2cc} we finally obtain that
$d_2^{0,1}([U_{i,x}])=\Lambda_{i,2}^{2,0}(c_i)$.
{Thus we can find the first Chern classes of the $Y_i$ in our spectral sequence
iff  $\Lambda_{i,2}^{2,0}(c_i)$ does not vanish. However, 
as we shall see  below, the map $\Lambda_{i,2}^{2,0}$ is not  injective in
general. }

The end of the section is devoted to the study of   the spectral sequences  of
NCP $\TT^2$-bundles with base $\TT^2$. In case where 
the underlying  function $f:\TT^2\to \TT$ is homotopic to a constant, we get a complete 
description by Remark \ref{rem-triv}. If $f$ is not homotopic to a constant, then 
the only part of the differential $d_2$ which does not vanish automatically is
$$d_2^{0,1}:H^0(\TT^2,K^1(\TT^2))\to H^2(\TT^2,\mathcal{K}_0(A)).$$ 
since we shall see below
that $H^0(\TT^2,\mathcal K_0(A))$ is given by the invariants in $K_0(\TT^2)$ under 
the action of $\ZZ^2\cong \pi(\TT^2)$, and hence is generated by the class $[1]$ of the unit.
Since this class trivially extends to a class in $K_0(A(X))$, it must vanish under any 
differential in the spectral sequence.

To proceed let us first remark that if
$F:[0,1]\times X\to \TT^{n(n-1)/2}$ is a homotopy between $f_0: X\to
 \TT^{n(n-1)/2}$ and $f_1: X\to \TT^{n(n-1)/2}$ and if $q:Y\to X$ is
 any $\TT^n$-bundle, then {$\big(Y\times [0,1]\big)* F^*(C^*(H_n))$} is a
     homotopy of $NCP$-$\TT^n$-bundles and thus according to
     \cite[Proposition 3.2]{ENO},  $Y* f_0^*(C^*(H_n))$ and $Y* f_1^*(C^*(H_n))$ are
     $\RKK$-equivalent.

For $X=\TT^2$, the classifying data are {$([Y\stackrel{q}{\to} \TT^2],f)$}, where
$f:\TT^2\to\TT$ is a continuous function. According to the previous
remark, we can replace $f$ by a homotopic function and thus we can
assume that there exist {integers $k$ and $l$} such that
$f(z_1,z_2)=z_1^kz_2^l$ for every $(z_1,z_2)$ in $\TT^2$.  {Let us compute}  $H^*(\TT^2,\mathcal{K}_0(A))$.  {We have
$\pi_1(\TT^2)\cong\ZZ^2$ with action of the
generators $(1,0)$ and $(0,1)$ of $\ZZ^2$ on $K^0(\TT^2)$ in the base
$([1],\beta)$  given by
the matrices $\left(\begin{smallmatrix} 1&k\\0&1\end{smallmatrix}\right)$ and
$\left(\begin{smallmatrix} 1&l\\0&1\end{smallmatrix}\right)$, 
respectively (see \cite[Proposition 5.2]{ENO})}.

Let us fix  a simplicial decomposition of $\TT^2$. Then 
since $\RR^2\to\TT^2$ is 
the classifying covering for $\ZZ^2$, the simplicial complex   $\SS(\RR^2)$
is a free resolution for $\ZZ[\ZZ^2]$ and hence for any abelian group
$M$ equipped with an action of $\ZZ^2$, the cohomology of the
{complex} $\operatorname{Hom}_{\ZZ^2}(\SS(\RR^2),M)$ is naturally isomorphic
to $H^*(\ZZ^2,M)$. Recall from \cite{FH} that for an abelian group
$M$ equipped with an action of $\ZZ^n$, the cohomology group
$H^*(\ZZ^n,M)$ can be computed recursively in the following way:
\begin{itemize}
\item For $n=0$ we have  that  
$H^0(\ZZ^n,M)\cong M$ and $H^k(\ZZ^n,M)=\{0\}$ for $k\geq 1$.
\item Let us consider the action of $\ZZ^{n-1}$ on $M$ using the $n-1$
 last factors of $\ZZ^n$. Then the action of the first factor of
 $\ZZ^n$ induces an action of $\ZZ$ on $H^k(\ZZ^{n-1},M)$ and there
 is an natural exact sequence
$$0\longrightarrow\operatorname{Coinv}_\ZZ\, H^k(\ZZ^{n-1},M) \longrightarrow
H^k(\ZZ^{n},M)\longrightarrow\operatorname{Inv}_\ZZ\, H^{k-1}(\ZZ^{n-1},M))\longrightarrow
0.$$ 
\end{itemize}
From this, it is straightforward to check that $H^n(\ZZ^n,M)$ is
naturally isomorphic to $\operatorname{Coinv}_{\ZZ^n}\,M$. {In the
case $M=\ZZ$ equipped with the trivial action of $\ZZ^n$, the
corresponding 
identification   $H^n(\ZZ^n,\ZZ)\cong
\operatorname{Coinv}_{\ZZ^n}\,\ZZ\cong\ZZ$ is given by pairing with
the fundamental class of $H_n(\ZZ^n,\ZZ)$.} Under the identification
$H_*(\ZZ^n,\ZZ)\cong H_*(\TT^n,\ZZ)$, this class can be viewed as the
fundamental class $[\TT^n]$ of $H_n(\TT^n,\ZZ)$.

{Combining  all this}, we are now in the position to describe the $d_2$
map of the spectral sequence derived from a NCP $\TT^2$-bundle $A(\TT^2)$
with  classifying data $([Y\stackrel{q}{\to} \TT^2], f)$. Let
$k$ be the greatest {common divisor} of the  winding numbers of the two components
of  $f$. We can assume that $k\neq 0$, otherwise $f$ is
homotopic to a constant map and thus $A(\TT^2)$ is $\RKK$-equivalent to
$C(Y)$. Then
\begin{itemize}
\item $$H^2(\TT^2,\mathcal{K}_0(A))\cong
 \operatorname{Coinv}_{\ZZ^2}\, K^0(\TT^2)\cong \ZZ/ k \ZZ\oplus
 \ZZ,$$  where the image in $\operatorname{Coinv}_{\ZZ^2}\,
 K_0(\TT^2)$ of the class $[1]\in K_0(\TT^2)$ 
 is a generator for $\ZZ/ k \ZZ$, and where the image
 of the Bott element  $\beta\in K_0(\TT^2)$ is a generator for $\ZZ$.
\item Up to this identification, {$d_2^{0,1}$ has range in
 $\ZZ/k\ZZ$ } and 
$$d_2^{0,1}([U_{i,z}])=\langle c_i,[\TT^2]\rangle\mod k,$$ for
$z=f(1,1)$, $i=1,2$, and where $c_i$ is the Chern class of the
$\TT$-bundle $Y/\TT^2_i\to X$.
\end{itemize}
\begin{remark}
\
\begin{enumerate}
\item In particular, for  {the function $f(z_1,z_2)=z_1$}, the $d_2$ map vanishes for any
{principal} $\TT^2$-bundle $Y\to \TT^2$. We actually do not know at this stage
 whether all bundles corresponding to the function $f(z_1,z_2)=z_1$ are $\RKK$-equivalent,
 so we cannot answer the general question, whether two NCP-bundles with isomorphic spectral sequences must be $\RKK$-equivalent. We plan to investigate this question in future work.
  \item {The above computation can  be carried out} for any compact
 oriented surface $X$. The reason  is that the fundamental group 
 $$\Gamma=\langle a_i,b_i :i=1,\dots n;\prod_{i=1,\dots
   n} [a_i,b_i]=1\rangle$$
   (where $n$ is {the} genus of $X$) of 
 such surface satisfies Poincar\'e duality and thus 
$$H^2(X,\mathcal{K}_0(A))\cong
 \operatorname{Coinv}_{\Gamma}\, K_0(A_x)\cong \ZZ/ k \ZZ\oplus
 \ZZ,$$
 where $k$ is the greatest common {divisor} of
 $\langle f,a_i \rangle,\, \langle f,b_i \rangle$, $i=1,\dots
   n$, {where for any $\gamma$ in $\Gamma$, the integer    $\langle
 f,\gamma \rangle$ is the winding number of $f\circ h$ for a 
map  $h:\TT\to X$ representing the element $\gamma$.}
\end{enumerate}
\end{remark}

\def\mathcs{{\normalshape\text{C}}^{\displaystyle *}}

\end{document}